\documentclass[a4paper,12pt]{amsart}
\usepackage{amsmath}                    
\usepackage{amscd,amsthm,a4wide}
\usepackage{amsfonts}                   
\usepackage{amssymb}                    
\usepackage[all]{xy}                    
\usepackage[latin1]{inputenc}           
\usepackage[bookmarks=true,colorlinks=true,naturalnames=true]{hyperref} 
\usepackage{pdfsync}                    
\usepackage{times}                      
\usepackage{mathtools}                  

\newtheorem{theorem}                 {Theorem}      [section]

\newtheorem{corollary}    [theorem]  {Corollary}
\newtheorem{lemma}        [theorem]  {Lemma}

\theoremstyle{definition}
\newtheorem{example}      [theorem]  {Example}
\newtheorem{remark}       [theorem]  {Remark}
\newtheorem{definition}   [theorem]  {Definition}

\numberwithin{equation}{section}

\let\oldmarginpar\marginpar
\renewcommand\marginpar[1]{\ \oldmarginpar[\raggedleft\footnotesize #1]{\raggedright\footnotesize #1}}

\renewcommand\arraycolsep{1pt}
\def \U{\text{\bf U}}

\def \fieldC{\mathbb{C}}

\def \d{\mathrm{d}}

\def \equi{\Leftrightarrow}
\def \impl{\Rightarrow}

\begin{document}

\title{A note on the spectral deformation of harmonic maps from the two-sphere into the unitary group}

\author{Maria João Ferreira}
\author{Bruno Ascenso Simões}

\address{\textit{MJF}, \textit{BAS}: Departamento de Matemática, Faculdade de Ciências, Universidade de Lisboa, Campo Grande, 1749-016 Lisboa, Portugal, and Centro de Matemática e Aplicações Fundamentais, Universidade de Lisboa, Av.\ Prof.\ Gama Pinto 2\\1649-003, Lisbon, Portugal}

\email{mjferr@ptmat.fc.ul.pt; bsimoes@ptmat.fc.ul.pt}

\thanks{This work was partially supported by Fundação para a Ciência e Tecnologia, PORTUGAL.}

\begin{abstract}
In \cite{FerreiraSimoesWood:09}, together with J. C. Wood, the authors gave a completely explicit formula for all harmonic maps from $2$-spheres to the unitary group $U(n)$ in terms of freely chosen meromorphic functions on $S^2$. The simplest harmonic maps are the isotropic ones. Using Morse theory Burstall and Guest \cite{BurstallGuest:97} showed that the harmonic maps come in clusters labeled by the isotropic ones. In this work, using the formula for harmonic maps aforementioned, we describe explicitly this procedure, showing how all harmonic maps can be built from the isotropic ones.
\end{abstract}


\maketitle

\thispagestyle{empty}

\section*{Introduction}

A harmonic map $\varphi : M \rightarrow N$ between two Riemannian manifolds is a critical point of a natural Energy functional given by $E(\varphi) = \int_M \| d \varphi \|^2 v_M$, where $\| d \varphi \| ^2$ denotes the  square norm of differential of $\varphi$ and $v_M$ the volume element of $M$. Its Euler-Lagrange equation is a second order elliptic partial differential equation of divergence type. When the target manifold $N$ is a Lie group $G$ the harmonic map equation is one example of an ``integrable'' partial differential equation, that is a differential equation which allows a zero-curvature formulation. In this setting the harmonic map equation allows a spectral deformation, that is, a one-parameter deformation depending on a parameter $\lambda \in S^1$ (the extended solution) \cite{Uhlenbeck:89}. This means that the given harmonic map may be regarded as a map into the loop group $\Omega G$ of based loops $\gamma:S^1\rightarrow G$ ($\gamma (1) = Id$). Therefore such an extended solution admits a Fourier expansion. When this Fourier series has finitely many terms one says that the corresponding harmonic map has finite uniton number. This article deals with harmonic maps of finite uniton number, from a Riemann surface $M$ into the unitary group $U(n)$. For such a map, Uhlenbeck showed that the extended solution, and so the harmonic map, admits a factorization into linear factors -- the so called unitons. This construction gives a systematic procedure of obtaining the harmonic map by starting with a constant map and adding successive factors. J.~C. Wood and the authors exploited this point of view and showed that all harmonic maps of finite uniton number can be explicitly written in terms of known meromorphic functions \cite{FerreiraSimoesWood:09}. Since the complex Grassmannian $G_{*}(\mathbb{C}^n)$ can be totally geodesically immersed in its group of isometries, this procedure may be refined to obtain a correspondent explicit description of all harmonic maps from a Riemann surface into $G_{*}(\mathbb{C}^n)$ with finite uniton number \cite{FerreiraSimoes:12}.

We remark that $\Omega U(n) = \Omega _{alg} U(n)$, the subspace of all algebraic loops with finite Fourier expansion. Using Morse theory, Burstall and Guest classified these harmonic maps \cite{BurstallGuest:97}. As it is well known the energy functional $E:\Omega U(n) \rightarrow \mathbb{R}$ is a Morse-Bott function and its critical points are precisely the homomorphisms $S^1 \rightarrow U(n)$. They come in conjugate classes $\Omega _{\gamma }= \left \{ g\gamma g^{-1}: g \in U(n) \right \}$, where $\gamma $ is a closed geodesic in $U(n)$ passing through the identity $\text{Id}$. They are the fixed points of the action of $S^1$ on $\Omega U(n)$ given by $(\mu \gamma)(\lambda )= \gamma (\mu \lambda )\gamma (\mu )^{-1}$ ($\gamma \in \Omega U(n)$, $\mu \in S^1$). Under the gradient flow of $E$ the extended solution is deformed to a simpler one, i. e. one taking values in a critical manifold. Therefore, starting with a harmonic map $\varphi: M \rightarrow U(n)$ and following the flow line of $\nabla E$ one obtains a one parameter family $(\varphi_t)_{t \in \left ]0,1 \right ]}$ of harmonic maps whose end point $\lim_{t \rightarrow 0} \varphi_t$ is a harmonic map whose extended solution is $S^1$ invariant. The purpose of this paper is to give an explicit description of the one parameter family $\varphi_t$ in terms of the meromorphic data used \cite{FerreiraSimoesWood:09} to built the initial harmonic map $\varphi$.

The paper is organized as follows: Section \ref{Section:HarmonicMapsAndExtendedSolutions} is a preliminar section where we introduce the basic material on harmonic maps and extended solutions. In Section \ref{Section:ExplicitFormulaeForHarmonicMapsIntoUn} we review the explicit formulas describing all harmonic maps of finite uniton number from a Riemann surface into the unitary group in terms of freely chosen meromorphic functions and their derivatives, presented in \cite{FerreiraSimoesWood:09}, \cite{FerreiraSimoes:12}. The main results of these paper are presented in Section \ref{Section:StatementOfMainResult} and their complete proofs in Section \ref{Section:ProofOfMainResult}.


\section{Harmonic maps and extended solutions}\label{Section:HarmonicMapsAndExtendedSolutions}


Let $M^2$ be a Riemann surface, $U(n)$ the unitary group equipped with its bi-invariant metric induced, via left translation, by the inner product $(A,B)=\text{tr} \,AB^*$ on its Lie algebra $\mathfrak{G} = \mathfrak{U(n)}$ of skew-hermitian matrices.

A smooth map $\varphi : M \rightarrow U(n)$ is harmonic if it is a critical point of the Euler-Lagrange (energy) functional $\varphi \rightarrow \int_{M} \|d \varphi\|^2 v_M $, where $d \varphi$ is the differential of the map $\varphi$ and $v_M$ is the volume element on $M$ induced by its metric.

For convenience we choose a local complex coordinate $z$ on an open subset of $M$; our theory will be independent of that choice.

The Euler-Lagrange equation for the energy functional is the following partial differential equation:
\begin{equation}\label{Equation:EulerLagrangeEquation}
(\varphi^{-1} \varphi_{\overline{z}})_z + (\varphi^{-1}\varphi_z)_{\overline{z}} = 0,
\end{equation}
equivalently
\begin{equation*}\label{Equation:ReformulationOfEulerLagrangeEquation}
d^{\star}A = 0,
\end{equation*}
where $A=\textstyle{\frac{1}{2}}\varphi^{-1}\d\varphi$ is the pull-back of the Maurer-Cartan form on $U(n)$. We let $A_z$ and $A_{\overline{z}}$ denote the $(1,0)$ and $(0,1)$ parts of the $\mathfrak{U(n)}$-valued one form $A$, respectively.

Uhlenbeck \cite{Uhlenbeck:89} showed that the equation \eqref{Equation:EulerLagrangeEquation} allows a spectral deformation, introducing the notion of extended solution. This is a family of smooth maps $\phi_{\lambda}:M \rightarrow U(n)$ depending smoothly on $\lambda \in S^1$ such that whose Maurer-Cartan form
\begin{equation*}
A_{\lambda } = \phi_{\lambda }^{-1} d \phi_{\lambda } \in \Omega ^{1}(M, \mathfrak{U(n)},
\end{equation*}
satisfies
\begin{equation*}
A_{\lambda } = (1-\lambda ^{-1})A_z + (1-\lambda)A_{\overline{z}}.
\end{equation*}
and $\phi_{-1}=\varphi$,  $\phi_1=\text{Id}$ (the identity element of $U(n)$). Moreover \cite{Uhlenbeck:89} $\varphi$ is harmonic if and only if there exists (at least locally) an extended solution $\phi_{\lambda}$ with $\phi_{-1}=\varphi$. Note that any two extended solutions for a harmonic map may differ by a function $Q: S^1 \rightarrow U(n)$ with $Q(1)=\text{Id}$. Since the inversion $U(n) \rightarrow U(n)$ ($g \rightarrow g^{-1}$) is an isometry, if $\varphi$ is harmonic, $\varphi^{-1}$ is again harmonic and, up to left translation, an extended solution for $\varphi^{-1}$ is $\phi_{-\lambda }\phi_{-1}^{-1}$ \cite{BurstallGuest:97}.

The map $\phi$ can be interpreted as a map into a loop group $\Omega U(n)$. We write $\phi : M \rightarrow \Omega U(n)$, $\phi_{\lambda} = \phi(\lambda)$.

Each loop $\gamma :S^1 \rightarrow U(n)$ is represented by a Fourier series $\gamma (\lambda ) = \sum_{j \in \mathbb{Z}} A_j \lambda ^j$ with coefficients $A_j \in End \mathfrak{U(n)}$, for some $k\in\mathbb{N}$. If we let $\Omega _k (U(n))$ denote the set of loops in $U(n)$ with Fourier series $\sum_{\mid j\mid \leq k} A_j \lambda ^j$, then
\begin{equation*}
\Omega (U(n)) = \cup _k \Omega _k (U(n)).
\end{equation*}
An extended solution is said to have finite uniton number if $\phi(M) \subset \Omega _k (U(n))$ for some $k$. As we have remarked above, the correspondence between harmonic maps and extended solutions is not one to one. If $\phi$ is an extended solution for a given harmonic map $\varphi : M \rightarrow U(n)$, the minimal $k$ such that $Q\phi (M) \subset \Omega _k (U(n))$, $Q \in \Omega (U(n))$, is called the uniton number of the harmonic map $\varphi$ and will be denoted by $r_{\varphi}$, or simply by $r$. Equivalently, \cite{Uhlenbeck:89},
\[r=\min \left \{k \in \mathbb{N}: Q \phi = \sum_{i=0}^{k}\lambda ^i A^i ~, ~Q \in \Omega (U(n))  \right \}.\]
Every harmonic map on a compact simply connected Riemann surface has finite uniton number \cite{Uhlenbeck:89}.

Let $\Lambda  Gl(\mathbb{C}^n)$ denote the space of all smooth maps $\gamma : S^1 \rightarrow Gl(\mathbb{C}^n)$ and $\Lambda _{+}$ its subgroup consisting of maps such that $\gamma $ extends holomorphically to the unit disc, i.e.
\[\Lambda _{+}= \{ \gamma :S^1 \rightarrow Gl(\mathbb{C}^n) : \gamma = \sum_{i\geq 0} \lambda^i A_i ~, ~\gamma ^{-1}= \sum_{i\geq 0}\lambda ^i B^i \}.\]
We refer to \cite{PressleySegal:86} for further developments on this theory. We only need the fact that $\Lambda Gl(\mathbb{C}^n) = \Omega (U(n))\Lambda _{+}$ (Iwasawa decomposition), where $\Omega (U(n))\cap \Lambda _{+}= \{\text{Id} \}$. Indeed, multiplication
\begin{equation*}
\Omega (U(n)) \times \Lambda _{+} \rightarrow \Lambda Gl(\mathbb{C}^n)
\end{equation*}
is a diffeomorphism \cite{PressleySegal:86}. Therefore each $\gamma \in \Lambda Gl(\mathbb{C}^n)$ may be written uniquely as
\begin{equation*}
\gamma =\gamma _u \gamma _{+},
\end{equation*}
where $\gamma _u \in \Omega (U(n))$ and $\gamma _{+} \in \Lambda _{+}$.

The Iwasawa decomposition yields an action of $\Lambda Gl(\mathbb{C}^n)$ on $\Omega (U(n))$ with isotropy subgroup $\Lambda _{+}$. Since $\Lambda Gl(\mathbb{C}^n)$ and $\Omega (U(n))$ are complex groups we obtain $\Omega (U(n)) = \Lambda Gl(\mathbb{C}^n)/\Lambda _{+}$ so that $\Omega (U(n))$ is endowed with a complex homogeneous structure. Another consequence of the Iwasawa decomposition is that each extended solution $\phi: M \rightarrow \Omega (U(n))$ may be written as $\phi = [\Psi ]$ for some $\Psi : M \rightarrow \Lambda Gl(\mathbb{C}^n)$, i.e. $\phi = \Psi _u$. Such $\Psi $ will be called a complex extended solution.

One has a natural action of $\mathbb{C}^* = \mathbb{C}\setminus \{0\}$ on $\Omega(U(n))$: namely,
\begin{equation}
\label{Equation:CStarAction}
(u \gamma)(\lambda ) =\gamma (\lambda u)\gamma (u)^{-1},
\end{equation}
for $u \in \mathbb{C}^*$, $\gamma \in \Omega (U(n))$. It is easily seen that, if $\phi : M \rightarrow \Omega (U(n))$ is an extended solution, $u \phi :M \rightarrow \Omega (Gl(\mathbb{C}^n))$ is a complex extended solution \cite{Uhlenbeck:89}. Clearly, $\phi_{\lambda}$ is a fixed point of this action if and only if, for each $z_0 \in M$, $\phi(z_0): S^1 \rightarrow U(n)$ is a homomorphism (closed geodesic), that is, a critical point for the energy functional on paths $E: \Omega (U(n)) \rightarrow \mathbb{R}$ ($E(\gamma ) = \int_{S^1} \|\gamma '\|^2$). The critical manifolds are the conjugacy classes of such homomorphisms, namely the classes $\Omega _{\gamma }= \left\{ g\gamma g^{-1}: g \in U(n) \right \}$. The harmonic maps arising this way are called isotropic (see \cite{EschenburgQuast}, \cite{EschenburgTribuzy}).

In this situation ($\Omega U(n) =\Omega _{alg}U(n)$), it turns out that the flow line of $\nabla E$ starting at some $\gamma \in \Omega( U(n))$ is defined for all time; it is given by the action of the sub-semigroup $]0,1]$ of $\mathbb{C}^*$ \cite{BurstallGuest:97}. Thus, given an extended solution $\phi$ and applying the gradient flow of the Morse-Bott function $E$ defined above, we associate a one parameter family of extended solutions $\left \{t \phi \right \}_{t \in ]0,1]}$ in such a way that $\lim_{t \rightarrow 0} (t \phi)$ is a $S^1$-invariant extended solution \cite{BurstallGuest:97}. In Section \ref{Section:StatementOfMainResult} we explicitly describe this procedure.


\section{Explicit formulae for Harmonic maps into $\U(n)$}\label{Section:ExplicitFormulaeForHarmonicMapsIntoUn}


Starting with a constant map, Uhlenbeck showed how to construct all harmonic maps $S^2 \rightarrow U(n)$ through a series of successive multiplications by suitable maps into Grassmannians -- the so called adding a uniton method. In \cite{FerreiraSimoesWood:09} J.~C.Wood and the authors showed how to explicitly build those unitons. To explain this we need some notation.

Let $\underline{\mathbb{C}}^n$ denote the trivial complex bundle equipped with the standard Hermitian inner product on each fiber, For each subbundle $\underline{\alpha}$ of $\underline{\mathbb{C}}^n$, let $\pi_{\alpha}$ and $\pi_{\alpha}^{\perp}$ denote the orthogonal projection onto $\underline{\alpha}$ and $\underline{\alpha}^{\perp}$, respectively, where $\underline{\alpha}^{\perp}$ represents the orthogonal complement of $\underline{\alpha}$. For each $\mathbb{C}^{n}$-valued meromorphic function $H$ on $M$ we denote by $H^{(k)}$ ($k\geq 0$) its $k$-th derivative with respect to some local complex coordinate on $M$.

The next theorem tells how all harmonic maps can be explicitly  built in terms of freely chosen meromorphic functions.

\begin{theorem}\label{Theorem:Theorem11FromFerSimWood09} For any $r\in\{0,1,\ldots, n-1\}$, let $(H_{i,j})_{0 \leq i \leq r-1,1\leq j\leq J}$ be an $r\times J$ ($J\leq n$) array of $\fieldC^n$-valued meromorphic functions on $M^2$, and let $\varphi_0$ be an element of $\U(n)$.
For each $i = 0,1,\ldots,r-1$, set $\underline{\alpha}_{i+1}$ equal to the subbundle of $\underline{\fieldC}^n$ spanned by the vectors
\begin{equation}\label{Equation:FirstEquationInTheorem:Theorem11FromFerSimWood09}
\alpha^{(k)}_{i+1,j}=\sum_{s=k}^{i} C^{i}_s H^{(k)}_{s-k, j}\qquad(j = 1,\ldots, J, \ k =0,1, \ldots, i).
\end{equation}
Then, the map  $\varphi:M^2 \to \U(n)$ defined by
\begin{equation}\label{Equation:SecondEquationInTheorem:Theorem11FromFerSimWood09}
\varphi=\varphi_0 (\pi_1-\pi_1^{\perp})\cdots(\pi_r-\pi_r^{\perp})
\end{equation}
is harmonic.

Further, all harmonic maps of finite uniton number, and so all harmonic maps from $S^2$, are obtained this way.
\end{theorem}

Here we have used the following notation: for each $i$, $\pi_i$ denotes $\pi_{\underline{\alpha}_i}$ whereas $\pi_i^{\perp}$ stands for $\pi_{\underline{\alpha}_i^\perp}$. Moreover, for integers $i$ and $s$ with $0\leq s\leq i$, $C^i_s$ denotes the \emph{$s$'th elementary function} of the projections $\pi_i^{\perp}, \ldots, \pi_1^{\perp}$ given by
\begin{equation}\label{Equation:DefinitionOfCSIWithoutTheParameter}
C^i_s=\sum_{1\leq i_1<\cdots<i_s\leq i}\pi_{i_s}^{\perp}\cdots\pi_{i_1}^{\perp}.
\end{equation}

\bigskip

The knowledge of the uniton number and initial data $(H_{i,j})$ completely describes the feature of the harmonic map.

The array $(H_{i,j})$ which determines a given list of $\underline{\alpha}_i$ is not unique; for example it can be replaced by any array with the same column span over the meromorphic functions. We will say that two arrays $(H_{i,j})$ and $(\overline{H}_{i,j})$ are equivalent if they determine the same harmonic map. Indeed, through column operations, one can replace the array by one equivalent with linearly independent columns in the \emph{echelon} form

\renewcommand{\arraycolsep}{1pt}
\begin{equation}\label{Equation:ArrayHInEchelonForm}
\mathfrak{H}=\left[
\begin{array}{cccccccccccc}
H_{0,1}  &...&H_{0,d_1}  & 0           &...& 0         & 0           & 0 &...        & 0           &...&0          \\
H_{1,1}  &...&H_{1,d_1}  &H_{1,d_1+1}  &...&H_{1,d_2}  & 0           & 0 &...        & 0           &...&0          \\
H_{2,1}  &...&H_{2,d_1}  &H_{2,d_1+1}  &...&H_{2,d_2}  &H_{2,d_2+1}  &...&H_{2,d_3}  &0            &...&0          \\
...      &...&...        & ...         &...& ...       & ...         &...& ...       &...          &...&...        \\
H_{r-1,1}&...&H_{r-1,d_1}&H_{r-1,d_1+1}&...&H_{r-1,d_2}&H_{r-1,d_2+1}&...&H_{r-1,d_3}&H_{r-1,d_3+1}&...&H_{r-1,d_r}\\
\end{array}
\right]
\end{equation}\renewcommand{\arraycolsep}{1pt}\noindent

with $0 \leq d_1 \leq d_2 \leq ...\leq d_r \leq n$, where, for each $i = 1, ..., r$, the sub-array made up of the first $i$ rows and the first $d_i$ columns has linearly independent columns; equivalently, for each $i$ the vectors $H_{i, d_i + 1}, ..., H_{i, d_{i+1}}$ are linearly independent. $\alpha_{i}^{(0)}$ is built from that sub-array by the formula \eqref{Equation:FirstEquationInTheorem:Theorem11FromFerSimWood09} and so has rank at most $d_i$.

In the sequel the above array will be referred as the initial data for the harmonic map $\varphi$; we will assume that an initial data is always given in an echelon form.

For a given array $\mathfrak{H}$ we denote by $\varphi^{\mathfrak{H}}$ the corresponding harmonic map, according to Theorem \ref{Theorem:Theorem11FromFerSimWood09}.

It can be easily seen that if the data $(H_{i,j})$ is in diagonal form

\renewcommand{\arraycolsep}{2.5pt}
\begin{equation}\label{Equation:HsInDiagonalForm}
\left[
\begin{array}{cccccccccccccc}
H_{0,1}&...& H_{0,d_1}& 0         &...&  0      & 0         &...& 0       & 0 &...&0                &...&0          \\
0      &...& 0        &H_{1,d_1+1}&...&H_{1,d_2}& 0         &...& 0       & 0 &...&0                &...&0          \\
0      &...& 0        & 0         &...& 0       &H_{2,d_2+1}&...&H_{2,d_3}& 0 &...&0                &...&0          \\
...    &...&...       & ...       &...& ....    & ...       &...& ...     &...&...&...              &...&...        \\
0      &...&0         &0          &...&0        &0          &...&0        & 0 &...&H_{r-1,d_{r-1}+1}&...&H_{r-1,d_r}\\
\end{array}
\right]
\end{equation}\renewcommand{\arraycolsep}{1pt}\noindent

then the unitons $\underline{\alpha}_i$ are nested, i.e. $\underline{\alpha}_i \subseteq \underline{\alpha}_{i+1}$. Therefore the harmonic map $\varphi^{\mathfrak{H}} = (\pi_{1}-\pi_{1}^{\perp})...(\pi_r - \pi_{r}^{\perp})$ has image in $G_{*}(\mathbb{C}^n)$, where $\pi_i = \pi_{\underline{\alpha}_i}$ and $\pi_{i}^{\perp} = \pi_{\underline{\alpha}_{i}^{\perp}}$ . These are the harmonic maps invariant by the $S^1$-action described above \cite{Uhlenbeck:89}. They are also called isotropic harmonic maps \cite{EschenburgTribuzy}.

\begin{remark}
Consider now the arrays $\mathfrak{H}=(H_{i,j})$ and $\overline{\mathfrak{H}}=(\overline{H}_{i,j})$, where $\mathfrak{H}$ is in diagonal form as in \eqref{Equation:HsInDiagonalForm} and, for each $d_k < j \leq d_{k+1}$ ~ $(k \in \mathbb{N})$ and $i > k$, $\overline{H}_{k,j} = H_{k,j}$ and $\overline{H}_{i,j} = a_{i,j}H_{k,j}$ ~$(a_{i,j} \in \mathbb{R})$, i.e.: $\overline{\mathfrak{H}}$ is given by

\renewcommand{\arraycolsep}{1pt}
\[\left[
\begin{array}{ccccccccccc}
H_{0,1}         &...& H_{0,d_1}          & 0                      &...&          0         & 0                      &...&0                &...&0\\
a_{1,1}H_{0,1}  &...&a_{1,d_1}H_{0,d_1}  & H_{1, d_1 + 1}         &...&H_{1, d_2}          & 0                      &...&0                &...&0\\
a_{2,1}H_{0,1}  &...&a_{2,d_1}H_{0,d_1}  &a_{2,d_1+1}H_{1,d_1+1}  &...&a_{2,d_2}H_{1,d_2}  & H_{2,d_2 + 1}          &...&0                &...&0\\
...             &...& ...                &            ...         &...&                    &...                     &...&...              &...&...\\
a_{r-1,1}H_{0,1}&...&a_{r-1,d_1}H_{0,d_1}&a_{r-1,d_1+1}H_{1,d_1+1}&...&a_{r-1,d_2}H_{1,d_2}&a_{r-1,d_2+1}H_{2,d_2+1}&...&H_{r-1,d_{r-1}+1}&...&H_{r-1,d_r}\\
\end{array}
\right]\]\renewcommand{\arraycolsep}{1pt}\noindent
Using a simple inductive argument it can be easily seen that the two arrays $\mathfrak{H}$ and $\overline{\mathfrak{H}}$ are equivalent and produce the same isotropic harmonic map. This fact will be used in the next section.
\end{remark}
In the sequel a given array will be called a diagonal array if it is equivalent to an array in diagonal form.

\begin{example}

It is well known \cite{Uhlenbeck:89} that the maximal uniton number for a harmonic map $\phi: S^2 \rightarrow U(n)$ is $n-1$. Those maps with maximal uniton number are built out of an initial array of the type

\renewcommand{\arraycolsep}{5pt}
\[\left[
\begin{array}{cccc}
H_{0,1}  & 0 & ... & 0  \\
H_{1,1}  & 0 & ... & 0  \\
H_{2,1}  & 0 & ... & 0  \\
...      & 0 &...  & 0\\
H_{n-2,1}& 0 &...  & 0,
\end{array}.
\right]\]\renewcommand{\arraycolsep}{1pt}\noindent
In this case, $\alpha^{(k)}_{i+1}=\sum_{s=k}^{i} C^{i}_s H^{(k)}_{s-k, 1}$ and $\underline{\alpha}_{i+1}=\text{span}\{\alpha_{i+1}^{(k)}\}_{\text{\tiny{$0\leq k\leq i$}}}$ for each $i\leq n-2$.
\end{example}

\begin{example}

Let us consider now the case $r=2$. We may now choose several non-null columns. For instance, take the array $\left( H_{ij}\right) $ consisting of two rows

\renewcommand{\arraycolsep}{5pt}
\[\left[
\begin{array}{ccccc}
H_{0,1} & ... & H_{0,d_1}  & 0 ...&0 \\
H_{1,1} & ... & H_{1,d_1} & 0 ...&0 \\
\end{array}
\right].\]\renewcommand{\arraycolsep}{1pt}\noindent
In this case, \underline{$\alpha $}$_{1}$ is the span of $\left\{H_{01},...,H_{0d_{1}}\right\} $ and \underline{$\alpha $}$_{2}$  is the span of $\ \left\{ H_{0j}+C_{1}^{1}H_{1j},C_{1}^{1}H_{0j}^{\prime } : 1 \leq j \leq d_1 \right\}$.

\end{example}

\begin{example}
Let $F_0$ denote a constant subspace of $\mathbb{C}^7$ with $\dim 3$, $L_0$, $L_1$ be linearly independent $F_0$-valued meromorphic functions and $E_0$ a $F_{0}^{\perp}$-valued meromorphic function. The following array:

\renewcommand{\arraycolsep}{5pt}
\[\mathfrak{H}=\left[
\begin{array}{ccc}
L_0 & E_0 & 0  \\
0 & 0 & L_1
\end{array}
\right]\]
\renewcommand{\arraycolsep}{1pt}\noindent

corresponds to a $S^1$-invariant harmonic map into $G_2(\mathbb{C}^7)$, $\varphi = Q(\pi_1 - \pi_{1}^{\perp})(\pi_{2}-\pi_{2}^{\perp})$, where $Q=\pi_{F_0} - \pi_{F_{0}^{\perp}}$.

Notice that left multiplication by a constant map $Q$ does not, in general, preserves the image in $G_*(\mathbb{C}^n)$. Letting $\varphi : M \rightarrow G_2(\mathbb{C}^7)$ denote the above harmonic map, $\varphi^{\mathfrak{H}} = (\pi_{F_0} - \pi_{F_{0}^{\perp}}) \varphi$ is not Grassmannian valued since $\varphi^{\mathfrak{H}} \times \varphi^{\mathfrak{H}}\neq \text{Id}$.

\end{example}

\begin{example}
Let $\varphi : M \rightarrow G_{*}(\mathbb{C}^n)$ be a non-constant harmonic map of uniton number $1$. Then, if $\varphi$ is not holomorphic, it must be of the form $\varphi = (\pi_{F_0} - \pi_{F_{0}^{\perp}})(\pi_1 - \pi_1 ^{\perp})$. It is easily seen that $\varphi$ is $G_{*}(\mathbb{C}^n)$-valued if and only if $\pi_{F_0}$ and $\pi_1$ commute, i.e. if and only if $F_0$ decomposes $\underline{\alpha}_1$

\begin{equation}
\underline{\alpha}_1 = \underline{\alpha}_1\cap F_0 \oplus \underline{\alpha}_{1}^{\perp}\cap F_{0}^{\perp}.
\end{equation}
\end{example}

General harmonic maps into Grassmannian manifolds can also be described in terms of meromorphic data. To describe that we need the following definition:

\begin{definition}
Let $F_0$ be a constant subspace in $\mathbb{C}^n$. An $r \times n$ $F_0$-array is a family of $\mathbb{C}^n$-valued meromorphic functions $(K_{i,j})_{0\leq i\leq r-1, 1\leq j\leq n}$ such that, for each $j$, either
\begin{equation}
\begin{array}{l}
\pi_{F_0^\perp}(K_{2k,j})=0\text{ and }\pi_{F_0}(K_{2k+1,j})=0\text{ or}\\
\pi_{F_0}(K_{2k,j})=0\text{ and }\pi_{F_0^\perp}(K_{2k+1,j})=0,
\end{array}
\end{equation}
where $0\leq k\leq\frac{r-1}{2}$.

\end{definition}

For $G_{*}(\mathbb{C}^n)$-valued harmonic maps theorem $2.1$ specializes in the following way \cite{FerreiraSimoes:12}:

\begin{theorem}\label{Theorem:MainTheoremNonBundleVersion}
Let $F_0$ be a constant subspace in $\fieldC^n$. For any $r\in\{0,1,...,n-1\}$, let \\ $(K_{i,j})_{0\leq i\leq r-1,1\leq j\leq n}$ be an $r\times n$ $F_0$-array of $\fieldC^n$-valued meromorphic functions on $M^2$. For each $j$, consider the meromorphic functions
\begin{equation}\label{Equation:FirstEquationInTheorem:MainTheoremNonBundleVersion}
\begin{array}{l}
H_{0,j}=K_{0,j}\text{ and}\\
H_{i,j}=\displaystyle{\sum_{s=1}^i(-1)^{s+i}\binom{i-1}{s-1}K_{s,j}\text{,  }i\geq 1}.
\end{array}
\end{equation}
For each $0\leq i\leq r-1$, set $\underline{\alpha}_{i+1}$ equal to the subbundle of $\fieldC^n$ spanned by the vectors
\[\alpha_{i+1,j}^{(k)}=\sum_{s=k}^i C^i_s H_{s-k,j}^{(k)},\,(j=1,...,n,\,k=0,...,i).\]
Then, the map $\varphi:M^2\to \U(n)$ defined by
\[\varphi=(\pi_{F_0}-\pi_{F_0}^\perp)(\pi_{1}-\pi_{1}^\perp)...(\pi_{r}-\pi_{r}^\perp)\]
is harmonic.

Further, all harmonic maps from $M^2$ to $G_*(\fieldC^n)$ of finite uniton number, and so harmonic maps from $S^2$ to $G_*(\fieldC^n)$ are obtained this way.
\end{theorem}


\section{The spectral parameter}\label{Section:StatementOfMainResult}


Our main result is based on the fact that the harmonic maps $S^2 \rightarrow U(n)$ come into clusters labeled by the  different isotropic harmonic maps (\emph{basic maps}). The flow lines of $\nabla E$ carry the harmonic maps to its corresponding basic harmonic map. The main purpose of this section is to explicitly describe this procedure in terms of our initial data.

\smallskip

Throughout this section we let
\begin{equation}
\varphi= Q (\pi_1 - \pi_{1}^{\perp})...(\pi_r - \pi_{r}^{\perp})
\end{equation}
be a harmonic map with uniton number $r$ built out of the array $\mathfrak{H}$ as in \eqref{Equation:ArrayHInEchelonForm}

As mentioned above, its extended solution $\phi$ can be regarded as a map $\phi: S^2 \rightarrow \Omega_{alg} U(n)$. Then, under the gradient flow of the energy functional, $E: \Omega U(n) \rightarrow \mathbb{R}$, $E(\gamma )=\int_{S^1}|\gamma'|$, $\phi$ can be deformed into a simpler extended solution $\phi_0$, one taking values in a conjugacy class of a Lie group homomorphism $S^1 \rightarrow \Omega  U(n)$. The corresponding harmonic map $\varphi_0$ is isotropic and has the same uniton number.

To describe this feature in terms of the initial data just consider the following one parameter family:
\begin{equation}
\mathfrak{H}(t)=(H_{i,j}(t)),
\end{equation}
where, for each $d_k + 1 \leq j \leq d_{k+1}$,
\begin{equation}\label{Equation:DefinitionOfHijt}
H_{i,j}(t)=\sum_{s=0}^{i}{i\choose s} t^{s-k}(t-1)^{i-s} H_{s,j}
\end{equation}
We recall that $d_k$ denotes the number of non zero entries in row $k$ and we will consider $d_0=0$. Clearly $\mathfrak{H}(1) = \mathfrak{H}$.

We now state the main theorem:

\begin{theorem}\label{Theorem:MainTheorem}
Let $\varphi^{\mathfrak{H}}:S^2\rightarrow U(n)$ be a harmonic map of uniton number $r$ corresponding to an array $\mathfrak{H}$. Then, the one parameter family of harmonic maps $\varphi_t^{\mathfrak{H}}$ corresponding to the flow line $\nabla E$ is associated to the one parameter family of arrays $\mathfrak{H}(t)$ .
\end{theorem}

We remark that, when the parameter $t$ goes to zero, the one parameter family of arrays $\mathfrak{H}(t)$ ends up in a diagonal array $\mathfrak{H}(0)$. Therefore, reversing the procedure, i.e., departing form an array $\mathfrak{H}_0$ in diagonal form, corresponding to an isotropic harmonic map with uniton number $r$, and filling in the entries of each non zero line below the diagonal, we obtain all harmonic maps with the same uniton which flow into $\varphi^{\mathfrak{H}}_0$. In this way we see how all the harmonic maps with uniton number $r$ can be explicitly built up from the isotropic harmonic maps with the same uniton number.

The extended solution of an isotropic harmonic map $\varphi^{\mathfrak{H}_0}$, up to a discrete subset of $S^2$, takes values in a conjugacy class $\Omega _{\gamma }$, for some geodesic $\gamma: S^1 \rightarrow U(n)$, where $\Omega_{\gamma }= \{ g \gamma g^{-1} : g \in G \}$ is a critical manifold of the Morse Bott function $E$. It is well known that $\Omega U(n) = \Omega_{alg}U(n)$ is a disjoint union of the so called "unstable manifolds" $U_{\gamma }$, where $U_{\gamma }$ is the domain of attraction of the critical manifold $\Omega _{\gamma}$ under the flow of $\nabla E$.

Burstall and Guest \cite{BurstallGuest:97} showed that, up to a discrete subset of $S^2$, every extended solution takes values in a single unstable manifold $U_{\gamma}$. Starting with an isotropic harmonic map $\varphi^{\mathfrak{H}_0}$ whose extended solution takes values almost everywhere in $\Omega _{\gamma}$ and denoting by $u_{\gamma}:U_{\gamma} \rightarrow \Omega _{\gamma}\subset \Omega $ the map assigning to each $\eta \in U_{\gamma}$ the corresponding end point of the flow line of $\nabla E$, we have $u_{\gamma} \circ \varphi^{\mathfrak{H}_t}= \varphi^{\mathfrak{H}_0}$ almost everywhere. Then the above procedure explicitly describes all harmonic maps which culminate into a given $\varphi^{\mathfrak{H}_0}$.

\begin{example}
Let us start with the array
\[\left[
\begin{array}{ccc}
H_{01} & 0 & 0 \\
0 & H_{12} & 0 \\
0 & 0 & H_{23}%
\end{array}%
\right]\]
corresponding to a basic harmonic map $\varphi =(\pi _{1}-\pi_{1}^{\perp })(\pi _{2}-\pi _{2}^{\perp })(\pi _{3}-\pi _{3}^{\perp })$ of
uniton number $3$, where
\[\begin{array}{ll}
\alpha _{1}&=\text{span}\left\{ H_{01}\right\},\\
\alpha _{2}&=\alpha _{1}\oplus\text{span}\left\{ \pi _{1}^{\perp }H_{01}^{\prime },\pi _{1}^{\perp}H_{12}\right\}\quad\text{  and}\\
\alpha _{3}&=\alpha _{2}\oplus \text{span}\left\{\pi_{2}^{\perp }\pi _{1}^{\perp }H_{01}'',\pi _{2}^{\perp }\pi_{1}^{\perp }H_{12}',\pi_{2}^{\perp }\pi _{1}^{\perp}H_{23}\right\}.
\end{array}\]
Now, take the array
\[\left[
\begin{array}{ccc}
H_{01} & 0 & 0 \\
H_{11} & H_{12} & 0 \\
H_{21} & H_{22} & H_{23}
\end{array}
\right]\]

corresponding to the harmonic map $\widetilde{\varphi }=(\pi _{\widetilde{\alpha }_{1}}-\pi _{\widetilde{\alpha }_{1}}^{\perp })(\pi _{\widetilde{\alpha }_{2}}-\pi _{\widetilde{\alpha }_{2}}^{\perp })(\pi _{\widetilde{\alpha }_{3}}-\pi _{\widetilde{\alpha }_{3}}^{\perp })$. $\widetilde{\varphi }$ ``flows'' into $\varphi$ through the one parameter family $\widetilde{\varphi }_{t}=(\pi _{\widetilde{\alpha }_{1}(t)}-\pi _{\widetilde{
\alpha }_{1}(t)}^{\perp })(\pi _{\widetilde{\alpha }_{2}(t)}-\pi _{%
\widetilde{\alpha }_{2}(t)}^{\perp })(\pi _{\widetilde{\alpha }_{3}(t)}-\pi
_{\widetilde{\alpha }_{3}(t)}^{\perp })$, where

\renewcommand{\arraystretch}{1.3}
\[\begin{array}{ll}
\widetilde{\alpha }_{1}(t)&=\alpha _{1},\\
\widetilde{\alpha }_{2}(t)&=\text{span}\left\{ H_{01}+\pi _{1}^{\perp }(tH_{11}+(t-1)H_{01}),\pi _{1}^{\perp
}H_{01}',\pi _{1}^{\perp }H_{12}\right\}\\
~&=\text{span}\left\{ H_{01}+t\pi _{1}^{\perp}H_{11},\pi _{1}^{\perp
}H_{01}',\pi _{1}^{\perp }H_{12}\right\}\\
\widetilde{\alpha}_{3}(t)&=\text{span}\left\{H_{01}+C_{1}^{2}H_{11}(t)+C_{2}^{2}H_{21}(t),C_{1}^{2}H_{01}'+C_{2}^{2}H_{11}'(t),C_{1}^{2}H_{12}+C_{2}^{2}H_{22}(t),C_{2}^{2}H_{01}'',C_{2}^{2}H_{23}\right\}\\
\end{array}\]\renewcommand{\arraystretch}{1.0}
We now have that\renewcommand{\arraystretch}{1.3}
\[\begin{array}{ll}~&H_{01}+C_{1}^{2}H_{11}(t)+C_{2}^{2}H_{21}(t)\\
=&H_{01}+C_{1}^{2}(tH_{11}+(t-1)H_{01})+C_{2}^{2}(t^2H_{21}+2t(t-1)H_{11}+(t-1)^2H_{01})\\
=&H_{01}+tC_{1}^{2}H_{11}+C_{2}^{2}(t^2H_{21}+t(t-1)H_{11})+\pi_2^\perp((t-1)H_{01}+\pi_1^\perp t(t-1)H_{11})\\
=&H_{01}+tC_{1}^{2}H_{11}+C_{2}^{2}(t^{2}H_{21}+t(t-1)H_{11})
\end{array}\]\renewcommand{\arraystretch}{1.0}
Analogously,\renewcommand{\arraystretch}{1.3}
\[\begin{array}{ll}
C_{1}^{2}H_{01}'+C_{2}^{2}H_{11}'(t)&=C_{1}^{2}H_{01}'+tC_{2}^{2}H_{11}'\quad\text{and}\\
C_{1}^{2}H_{12}+C_{2}^{2}H_{22}(t)&=C_{1}^{2}H_{12}+tC_{2}^{2}H_{22}
\end{array}\]\renewcommand{\arraystretch}{1.3}
so that
\[\begin{array}{ll}\widetilde{\alpha}_{3}(t)&=\textrm{span}\left\{H_{01}+tC_{1}^{2}H_{11}+C_{2}^{2}(t^{2}H_{21}+t(t-1)H_{11}),\right.\\
~&\quad\left. C_{1}^{2}H_{01}'+tC_{2}^{2}H_{11}',C_{1}^{2}H_{12}+tC_{2}^{2}H_{22},C_{2}^{2}H_{01}'',C_{2}^{2}H_{23}\right\}
\end{array}\]\renewcommand{\arraystretch}{1.0}

Of course when $t$ goes to $0$ we go back to the initial basic harmonic map.
\end{example}

\begin{remark}

Regarding the $G_{\star }(\mathbb C^n)$ valued harmonic maps $\varphi=(\pi_{F_0}-\pi_{F_{0}^{\perp}})(\pi_1
- \pi_{1}^{\perp})...(\pi_r - \pi_{r}^{\perp})$ with uniton number $r$ we recall \cite{FerreiraSimoes:12}  that the entries of an inicial array $\mathfrak{H}$ for this harmonic map are built up from an $r \times n$ $F_0$-array of meromorphic functions $\mathfrak{K}^{\varphi}= (K_{i,j})$, where $1 \leq i \leq r-1$ and $1 \leq j \leq n$,   according to the rule

\begin{equation}
\begin{array}{l}
H_{0,j}=K_{0,j}\text{ and}\\
H_{i,j}=\displaystyle\sum_{s=1}^i(-1)^{s+i}\binom{i-1}{s-1}K_{s,j}\text{,  }i\geq 1.
\end{array}
\end{equation}

Now it is a matter of standard algebraic computations to see that our array $\mathfrak{H}(t)$ is equivalent to the array $\widetilde{\mathfrak{H}}(t)$  built out of the family $\mathfrak{K}(t)=(K_{i,j}(t))$ according with the rule described above, where $K_{i,j}(t)=t^{i-k}K_{i,j}$ for $d_k \leq j < d_{k+1}$.

Of course, when $\varphi$ has values in $G_p(\mathbb C^n)$ all the one parameter family $\varphi_t$ is also $G_p(\mathbb C^n)$--valued.
\end{remark}

\section{Proof of Theorem \ref{Theorem:MainTheorem}}\label{Section:ProofOfMainResult}


Let $\varphi^{\mathfrak{H}}:S^2\rightarrow U(n)$ be a harmonic map of uniton number $r$, built out of an array $\mathfrak{H}$ as in \eqref{Equation:ArrayHInEchelonForm}.

As before, we consider its complex extended solution $\phi^{\mathfrak{H}}:S^2\rightarrow \Omega(Gl(\mathbb{C}^n))$. Following the flow lines of the energy $E:\Omega(Gl(\mathbb{C}^n))\to\mathbb{R}$, the flow $\phi_t:S^2\rightarrow Gl(\mathbb{C}^n)$ starting at $\phi_1=\phi^{\mathfrak{H}}$ is given by $t\phi^{\mathfrak{H}}$ (see \eqref{Equation:CStarAction}).

From \cite{Uhlenbeck:89}, we know that $t\phi^{\mathfrak{H}}$ is an extended solution. Since the diffeomorphism
\begin{equation*}
\Omega (U(n)) \times \Lambda _{+} \rightarrow \Lambda Gl(\mathbb{C}^n)
\end{equation*}
preserves extended solutions \cite{DPW}, $\phi_t=(t\phi^{\mathfrak{H}})_u$ is an extended solution of a harmonic map $\varphi_t :S^2\rightarrow U(n)$. To conclude that $\varphi_t$ is the harmonic map built out of the array $\mathfrak{H}(t)$ as in Theorem \ref{Theorem:MainTheorem}, we just have to show that $\left(\phi^{\mathfrak{H}(t)}\right)^{-1} \phi_t \in \Lambda_{+}$,
where $\phi^{\mathfrak{H}(t)}$ denotes the extended solution associated to $\varphi^{\mathfrak{H}(t)}$.

To see this we consider $\left(\phi^{\mathfrak{H}(t)}\right)^{-1} \phi_t = \eta_r$ and, for each integer $k$, $\eta_{r}^{k}$ will denote the coefficient of $\lambda^k$. Everything will be settled proving that $\eta_{r}^{k}=0$ if $k<0$.

Before going into a detailed proof of this theorem, we consider first the cases $r=1$ and $r=2$.

\begin{example}\label{Example:UnitonNumberOneExample}
Let $\varphi:S^2\to U(n)$ be a harmonic map of uniton number $1$. Then, $\varphi=\pi_1-\pi_1^\perp$, where $\alpha_1=\text{span}\{H_{0j}\}$. Since $H_{0j}(t)=H_{0j}$, we get $\underline{\alpha}_1(t)=\underline{\alpha}_1$ and therefore $\pi_{\underline{\alpha}_1(t)}=\pi_{1,t}=\pi_1$ so that $\varphi(t)\equiv \varphi$. In particular, as it is well-known, uniton number one harmonic maps are $S^1$-invariant.

Also notice that
\[\hat{\Phi}_{\lambda^{-1}}\Phi_{\lambda t}=(\hat{\pi}_1+\lambda^{-1}\hat{\pi}_1^\perp)(\pi_1+\lambda t\pi_1^\perp)=(\pi_1+\lambda^{-1}\pi_1^\perp)(\pi_1+\lambda t\pi_1^\perp)=\pi_1+t\pi_1^\perp\]
has no negative powers of $\lambda$. We conclude that $\eta_1=\eta_1^{0}$. Hence, the first case is proved.

\end{example}

From now on, we will use the notation $\pi_{it}$ and $\pi_{it}^\perp$ to denote the orthogonal projections $\pi_{\underline{\alpha}_i(t)}$ and $\pi_{\underline{\alpha}_i^\perp(t)}$, respectively. In the same way, given the subbundles $\underline{\alpha}_i(t)$, we use the notation $C^i_s(t)$ ($s\leq i$) to denote the $s$'th elementary function of $\pi_{it}$.

\begin{example}\label{Example:UnitonNumberTwoExample}
For the uniton-two case, we start with $\underline{\alpha}_1$ as above and $\underline{\alpha}_2=\text{span}\{H_{0j_1}+\pi_1^\perp H_{1j_1},\pi_1^\perp H_{0j_1}',\pi_1^\perp H_{1j_2}\}$. Considering\renewcommand{\arraystretch}{1.3}
\[\begin{array}{ll}
\underline{\alpha}_2(t)&=\text{span}\{H_{0j_1}(t)+\pi_{1t}^\perp H_{1j_1}(t),\pi_{1t}^\perp H_{0j_1}'(t),\pi_{1t}^\perp H_{1j_2}(t)\}\\
~&=\text{span}\{H_{0j_1}+\pi_1^\perp((t-1)H_{0j_1}+tH_{1j_1}),\pi_1^\perp H_{0j_1}',\pi_1^\perp H_{1j_2}\}.\end{array}\]\renewcommand{\arraystretch}{1.0}

We have that\renewcommand{\arraystretch}{1.3}
\[\begin{array}{lll}
~     & \eta_2&=(\pi_{2,t}+\lambda^{-1}\pi_{2,t}^\perp)\eta_1^{0}(\pi_2+\lambda t\pi_2^\perp)\\
\impl & \eta_2^{-1}&=\pi_{2,t}^\perp\eta_1^0\pi_2,\,\eta_2^{0}=\pi_{2,t}\eta_1^0\pi_2+t\pi_{2,t}^\perp\eta_1^0\pi_2^\perp\,\text{and}\,\eta_2^1=t\pi_{2,t}^\perp\eta_1^0\pi_2^\perp.
\end{array}\]\renewcommand{\arraystretch}{1.0}

Notice that for any pair of functions $(V_m,V_{m+1})$,\renewcommand{\arraystretch}{1.3}
\begin{equation}\label{Equation:FirstEquationForTheUnitonOneCase}
\begin{array}{ll}
\eta_1^0(V_m+\pi_1^\perp V_{m+1})&=(\pi_1+t\pi_1^\perp)(V_m+\pi_1^\perp V_{m+1})=(\pi_1+t\pi_1^\perp)V_m+t\pi_1^\perp V_{m+1}\\
~&=V_m+\pi_1^\perp\big((t-1)V_m+tV_{m+1}\big)
\end{array}
\end{equation}\renewcommand{\arraystretch}{1.0}

Therefore, from \eqref{Equation:FirstEquationForTheUnitonOneCase}, we have that\renewcommand{\arraystretch}{1.3}
\[\begin{array}{rl}
\pi_{2,t}^\perp\eta_1^0 (H_{0j}+\pi_1^\perp H_{1j})&=\pi_{2,t}^\perp\big(H_{0j}+\pi_1^\perp\big((t-1)H_{0j}+tH_{1j}\big)\big)=0\,\,\text{and}\\
\pi_{2,t}^\perp\eta_1^0 (\pi_1^\perp H_{0j}')           &=\pi_{2,t}^\perp(\pi_1^\perp tH_{0j}')=0,
\end{array}\]
\renewcommand{\arraystretch}{1.0}

so that $\eta_2^{-1}=0$ and thus showing that $\eta_2\in\Lambda_+$.
\end{example}

Given any family $(V_{m},...,V_{m'})$ of $\mathbb{C}^n$-valued meromorphic functions defined on $M^2$ we define an auxiliary 1-parameter family of $\mathbb{C}^n$ valued maps by
\begin{equation}\label{Equation:DefinitionOfVijkt}
V_{i}^k(t)=\underbrace{\sum_{l=0}^k{k\choose l}(t-1)^{k-l}t^lV_{i+l}}_{\text{\tiny{starts in $V_{i}$ and ends in $V_{i+k}$}}},\,i\in\mathbb{Z},\,k\geq 0
\end{equation}

where we interpret $V_{ij}=0$ if $i<m$ or $i>m'$.

Notice that when the family $(V_m,...,V_{m'})$ is the $j$-th column $(H_{0,j},....,H_{i,j})$ ($d_0<j\leq d_1$) of $\mathfrak{K}$, then we have that, in this notation,
\[H_{0,j}^i(t)=\sum_{l=0}^i{i\choose l}(t-1)^{i-l}t^lH_{l,j}\]
which is exactly our $H_{i,j}(t)$ (see \eqref{Equation:DefinitionOfHijt}). We introduce this auxiliary family of $V$'s as the proof of Theorem \ref{Theorem:MainTheorem} relies on an algebraic Lemma valid for any family of $\mathbb{C}^n$ valued functions.

For the array $(V_{m},...,V_{m'})$ we define, for $k=0,1$,
\begin{align}
\beta_{i+1}^{k}&=\displaystyle{\sum_{s=0}^i C_{s}^i V_{m+s+k}},\quad\text{as well as} \label{Equation:DefinitionsOfBetas:Equation1}\\
\hat{\beta}_{i+1}^{k}(t)&=\displaystyle{\sum_{s=0}^i C_{s}^i(t) V_{m+k}^{s}(t)}\quad\text{and} \label{Equation:DefinitionsOfBetas:Equation2}\\
\beta_{i+1}^{k}(t)&=\displaystyle{\sum_{s=0}^i C_{s}^i(t) V_{m}^{s+k}(t)}\label{Equation:DefinitionsOfBetas:Equation3}
\end{align}
where again we interpret $V_{i}=0$ if $i<m$ or $i>m'$. We also notice that the $C_s^i(t)$ are defined as in \eqref{Equation:DefinitionOfCSIWithoutTheParameter} but with the underlying subbundles $\underline{\alpha}_i$ replaced by $\underline{\alpha}_i(t)$.

If $(V_{m},...,V_{m'})=(H_{0j},...,H_{ij})$ (eventually, some entries null if $j>d_1$), then
\begin{equation}\label{Equation:CaseVmVmEqualsHoHi}
\begin{array}{rl}
\beta_{i+1}^0&=\alpha_{i+1,j}^{(0)},\quad\text{and}\\
\hat{\beta}_{i+1}^0(t)&=\beta_{i+1}^0(t)=t^k\alpha_{i+1,j}^{(0)}(t),\,d_k <j \leq d_{k+1}
\end{array}
\end{equation}

Moreover, if $(V_{m},...,V_{m'})=(0_0,...,0_{k-1},H_{0j}^{(k)},...,H_{i-k,j}^{(k)})$, then
\begin{equation}\label{Equation:CaseVmVmEqualsHoHiCaseDerivatives}
\beta_{i+1}^0=\alpha_{i+1,j}^{(k)}
\end{equation}

Besides, we also have the following Lemma, which proof is presented in Section \ref{Section:Appendix}.
\begin{lemma}\label{Lemma:LemmaToProveTheAzCase} For $d_0< j\leq d_1$,
\begin{equation}\label{Equation:EqInLemmaToProveTheAzCase}
\beta_{i}^0(t)(0_0,...,0_{k-1},H_{0j}^{(k)},...,H_{i-k-1,j}^{(k)})=t^k\alpha_{ij}^{(k)}(t).
\end{equation}
More generally, for $d_l< j\leq d_{l+1}$,
\[\beta_{i}^0(t)(0_0,...,0_{k-1},\underbrace{H_{0j}^{(k)},...,H_{l-1,j}^{(k)}}_{\text{\tiny{$=0$}}},H_{lj}^{(k)}...,H_{i-k-1,j}^{(k)})=t^{k-l}\alpha_{ij}^{(k)}(t).\]
\end{lemma}

The idea to prove Theorem \ref{Theorem:MainTheorem} is to establish an induction that allows one to conclude that $\eta$ has no negative powers of $\lambda$. More precisely, we shall establish the following Lemma, which we prove in Section \ref{Section:Appendix}.
\begin{lemma}\label{Lemma:LemmaOfEquation:FirstHypothesisEquationVersionTwo}
For any $r'$, if $\eta_{r'}$ has no $\lambda$ negative powers, there are  $\text{End}(\mathbb{C}^n)$--valued $1$-forms $A^r_j$ (see \eqref{Equation:DefinitionOfArk}) such that for any fixed collection $(V_m,..., V_{m+r'})$,
\begin{equation}\label{Equation:FirstHypothesisEquationVersionTwo}
\eta_{r'}^0(\beta_{r'+1}^0)=\beta_{r'+1}^0(t)+\sum_{s=0}^{r'}\pi_{s,t}^\perp\beta_s^0(t)-\sum_{s=2}^{r'}A_s^{r'}\pi_s^\perp\beta_s^0.
\end{equation}
\end{lemma}

\textit{Proof of Theorem \ref{Theorem:MainTheorem}.}

Since
\[\eta_{r'+1}=(\pi_{r'+1,t}+\lambda^{-1}\pi_{r'+1,t}^\perp)\eta_{r'}(\pi_{r'+1}+\lambda t\pi_{r'+1}^\perp),\]
if $\eta_{r'}$ has no negative powers of $\lambda$, we shall have that
\[\eta_{r'+1}^{-1}=\pi_{r'+1,t}^\perp\eta_{r'}^0\pi_{r'+1}\]
so that $\eta_{r'+1}^{-1}$ will vanish as long as $\eta_{r'}^0$ maps $\underline{\alpha}_{r'+1}$ into $\underline{\alpha}_{r'+1}(t)$.

But applying \eqref{Equation:FirstHypothesisEquationVersionTwo} when $(V_m,...,V_m')=(H_{0j},...,H_{r'j})$ together with \eqref{Equation:CaseVmVmEqualsHoHi} yields that $\eta_{r'}^0$ maps $\alpha_{r'+1,j}^{(0)}$ into $\alpha_{r'+1,j}^{(0)}(t)$ since, in that case, \eqref{Equation:FirstHypothesisEquationVersionTwo} reduces to $\eta_{r'}^0(\beta_{r'+1}^0)=\beta_{r'+1}^0(t)$.

Applying \eqref{Equation:FirstHypothesisEquationVersionTwo} when $(V_m,...,V_m')=(0_0,...,0_{k-1},H_{0j}^{(k)}...,H_{r'-k,j}^{(k)})$ together with \eqref{Equation:CaseVmVmEqualsHoHiCaseDerivatives} and Lemma \ref{Lemma:LemmaToProveTheAzCase} shows that $\eta_{r'}^0$ maps $\alpha_{r'+1,j}^{(k)}$ into $\alpha_{r'+1,j}^{(k)}(t)$, since once again \eqref{Equation:FirstHypothesisEquationVersionTwo} reduces to $\eta_{r'}^0(\beta_{r'+1}^0)=\beta_{r'+1}^0(t)$. Hence, this concludes the proof of Theorem \ref{Theorem:MainTheorem}.

\qed

Therefore, we are now left with showing Lemmas \ref{Lemma:LemmaToProveTheAzCase} and \ref{Lemma:LemmaOfEquation:FirstHypothesisEquationVersionTwo}.


\section{Proof of Lemmas \ref{Lemma:LemmaToProveTheAzCase} and \ref{Lemma:LemmaOfEquation:FirstHypothesisEquationVersionTwo}}\label{Section:Appendix}


\begin{lemma}\label{Lemma:LemmaToProve:LemmaToProveTheAzCase} For a given $i$, we have that
\begin{equation}\label{Equation:EquationIn:Lemma:LemmaToProve:LemmaToProveTheAzCase}
\sum_{s=0}^lC^{i-1}_{i-l-1+s}(t)H_{st}^{(k)}=0,\quad\text{for all}\, 0\leq l\leq i-1
\end{equation}
\end{lemma}
\begin{proof} We prove \eqref{Equation:EquationIn:Lemma:LemmaToProve:LemmaToProveTheAzCase} by induction on $l$. When $l=0$, we get
\[C^{i-1}_{i-1}(t)H_{0t}^{(k)}=\pi_{i-1,t}^\perp\pi_{i-2,t}^\perp...\pi_{k+1,t}^\perp C^k_k(t)H_{0t}^{(k)}=\pi_{i-1,t}^\perp\pi_{i-2,t}^\perp...\pi_{k+1,t}^\perp \alpha_{k+1}^{(k)}(t)=0\]

Next, assume that \eqref{Equation:EquationIn:Lemma:LemmaToProve:LemmaToProveTheAzCase} is true and let us show that

\[\sum_{s=0}^{l+1}C^{i-1}_{i-l-2+s}(t)H_{st}^{(k)}=0.\]
\[\begin{array}{ll}
~&\displaystyle{\sum_{s=0}^{l+1}C^{i-1}_{i-l-2+s}(t)H_{st}^{(k)}=\sum_{s=0}^lC^{i-1}_{i-l-2+s}(t)H_{st}^{(k)}+C^{i-1}_{i-1}H_{l+1,t}^{(k)}}\\
=&\displaystyle{\underbrace{\sum_{s=0}^lC^{i-2}_{i-1-l-1+s}(t)H_{st}^{(k)}}_{\text{\tiny{$=0$, from hypothesis}}}+\pi_{i-1,t}^\perp\Big(\sum_{s=0}^{l}C^{i-2}_{i-l-3+s}(t)H_{st}^{(k)}+C^{i-2}_{i-2}(t)H_{l+1,t}^{(k)}\Big)}\\
=&...=\displaystyle{\pi_{i-1,t}^\perp\pi_{i-2,t}^\perp...\pi_{k+l+2,t}^\perp\Big(\sum_{s=0}^lC^{k+l+1}_{i-l+s-(i-k-l-1)-1}(t)H_{st}^{(k)}+C^{k+l+1}_{k+l+1}(t)H_{l+1,t}^{(k)}\Big)}\\
=&\pi_{i-1,t}^\perp\pi_{i-2,t}^\perp...\pi_{k+l+2,t}^\perp\alpha_{k+l+2}^{(k)}(t)=0,
\end{array}\]
\end{proof}

\textit{Proof of Lemma \ref{Lemma:LemmaToProveTheAzCase}}

We start by noticing that
\begin{equation}\label{Equation:EquationBeforeLemmaToProveTheAzCase}
\begin{array}{l}
\beta_{i+1}^0(t)(V_0,...,V_i)=\beta_{i-1}^0(t)(V_0,...,V_{i-1})+\pi_{it}^\perp\big(t\beta_i^0(t)(V_1,...,V_i)+(t-1)\beta_i^0(t)(V_0,...,V_{i-1})\big)
\end{array}
\end{equation}

We shall prove the lemma by induction on $i$ for the case $d_0<j\leq d_1$. The general case $d_l< j\leq d_{l+1}$ will follow easily. For $i=1$, we have that
\[\beta_{1}^0(t)(H_0)=H_0=t^0\alpha_1^{(0)}(t).\]
Next, assume that \eqref{Equation:EqInLemmaToProveTheAzCase} holds up to order $i$ and let us prove that
\[\beta_{i+1}^0(t)(\underbrace{0,...,0}_{\text{\tiny{$k$}}},H_{0}^{(k)},...,H_{i-k}^{(k)})=t^{k}\alpha_{i+1}^{(k)}(t)\]
As a matter of fact, using \eqref{Equation:EquationBeforeLemmaToProveTheAzCase},

\renewcommand{\arraystretch}{1.3}

$
\begin{array}{ll}
~&\beta_{i+1}^0(t)(0,...,0,H_{0}^{(k)},...,H_{i-k}^{(k)})-t^{k}\alpha_{i+1}^{(k)}(t)\\
=&\beta_i^0(t)(\underbrace{0,...,0}_{\text{\tiny{$k$}}},H_0^{(k)},...,H_{i-k-1}^{(k)})+t\pi_{it}^{\perp}\beta_i(t)(\underbrace{0,...,0}_{\text{\tiny{$k-1$}}},H_0^{(k)},...,H_{i-k}^{(k)})\\
~&\displaystyle{+(t-1)\pi_{it}^\perp\beta_i(t)(\underbrace{0,...,0}_{\text{\tiny{$k$}}},H_0^{(k)},...,H_{i-k-1}^{(k)})-t^k\sum_{s=k}^{i-1}C_s^{i-1}H_{s-k,t}^{(k)}-t^k\pi_{it}^\perp\sum_{s=k-1}^{i-1}H_{s-k+1,t}^{(k)}}
\end{array}
$
\renewcommand{\arraystretch}{1.0}
\medskip

Now, the third parcel vanishes by induction and the first cancels out with the forth, also using the induction hypothesis. Hence, we are left with\medskip

$
\begin{array}{ll}
~&\displaystyle{t\pi_{it}^\perp\Big(\beta_{i}^0(t)(0,...,0,H_0^{(k)},...,H_{i-k}^{(k)})-t^{k-1}\sum_{s=k-1}^{i-1}C_s^{i-1}(t)H_{s-k+1,t}^{(k)}\Big)}\\
=&\displaystyle{t\pi_{it}^\perp\Big(\sum_{s=k-1}^{i-1}C_s^{i-1}(t)\sum_{l=0}^{s-(k-1)}{s\choose s-k+1-l}t^{k-1+l}(t-1)^{s-k+1-l}H_{l}^{(k)}}\\
~&\displaystyle{-t^{k-1}\Big(\sum_{s=k-1}^{i-1}\sum_{l=0}^{s-k+1}{s-k+1\choose s-k+1-l}t^l(t-1)^{s-k+1-l}H_{l}^{(k)}\Big)\Big)}
\end{array}
$\medskip

Now, all the first terms that do not have $(t-1)$ as a factor cancel. Hence, we get
\smallskip

$
\begin{array}{ll}
=&\displaystyle{t^k(t-1)\pi_{it}^\perp\Big(\sum_{s=k}^{i-1}C_s^{i-1}(t)\sum_{l=0}^{s-k}{s\choose s-k+1-l}t^l(t-1)^{s-k-l}H_{l}^{(k)}}\\
~&\displaystyle{-\sum_{s=k}^{i-1}C_s^{i-1}(t)\sum_{l=0}^{s-k}{s-k+1\choose s-k+1-l}t^l(t-1)^{s-k-l}H_l^{(k)}\Big)}
\end{array}
$\medskip

Next, noticing that
\[{i-j\choose l}=\sum_{s=0}^k{k\choose s}{i-(k+j)\choose l-s},\]
(where, as usual, ${i\choose j}=0$ if $j>i$ or $j<0$), the last expressions becomes

\medskip
\renewcommand{\arraystretch}{1.3}
$
\begin{array}{ll}
=&t^k(t-1)\pi_{it}^\perp\big(C_k^{i-1}(t){k\choose 1}H_0^{(k)}+C^{i-1}_{k+1}(t)(({k\choose 0}+{k\choose 1})tH_1^{(k)}+({k\choose 1}+{k\choose 2})(t-1)H_0^{(k)})+...\\
~&+C^{i-1}_{i-1}(t)(({k\choose 0}{i-k-1\choose 1}+{k\choose 1}{i-k-1\choose 0})t^{i-k-1}H_{i-k-1}^{(k)}+...\\
~&+({i-k-1\choose i-k-1}{k\choose 1}+{i-k-1\choose i-k-2}{k\choose 2}+...+{i-k-1\choose i-2k}{k\choose k})(t-1)^{i-k-1}H_0^{(k)})\\
~&-\big(C^{i-1}_k(t)H_0^{(k)}+C^{i-1}_{k+1}(t)((1+1)tH_1^{(k)}+(t-1)H_0^{(k)})+...\\
~&+C^{i-1}_{i-1}(t)(({i-k-1\choose 0}+{i-k-1\choose 1}t^{i-k-1}H_{i-k-1}^{(k)}+...\\
~&\quad+({i-k-1\choose i-k-2}+{i-k-1\choose i-k-1})H_{1}^{(k)}+{i-k-1\choose i-k-1}(t-1)^{i-k-1}H_0^{(k)})\big)
\big)
\end{array}
$\renewcommand{\arraystretch}{1.0}
\medskip

Now, look first to the positive terms. The parcels that start with ${k\choose 1}$ give precisely ${k\choose 1}\alpha_i^{(k)}$. The negative parcels give $\alpha_i^{(k)}$ plus a remainder, that cancels precisely with the positive parcels that start with ${k\choose 0}$. Hence, we get ${k\choose 1}\alpha_i^{(k)}-\alpha_i^{k}$ plus all the positive parcels that start with ${k\choose j}$ with $j\geq 2$, all of which have $(t-1)$ as a factor. We are hence left with
\smallskip

$
\begin{array}{ll}
=&\displaystyle{t^k(t-1)^2\pi_{it}^\perp\Big(\sum_{s=k+1}^{i-1}C_s^{i-1}\sum_{l=0}^{s-k-1}\Big(\sum_{j=2}^{s-k-l+1}{k\choose j}{s-k\choose s-k-l-j+1}\Big)t^l(t-1)^{s-k-l-1}H_l^{(k)}\Big)}
\end{array}
$
\smallskip

Repeating the argument, we end up  with

\smallskip

$
\begin{array}{ll}
=&\displaystyle{t^k(t-1)^2\pi_{it}^\perp\Big(\sum_{j=2}^{i-k}{k+j-2\choose j}(t-1)^{j-2}\sum_{s=k+j-1}^{i-1}C_s^{i-1}(t)H_{s-k-j+1,t}\Big)}
\end{array}
$
\smallskip

which vanishes, by Lemma \ref{Lemma:LemmaToProve:LemmaToProveTheAzCase}, finishing the proof for the case $d_0<j\leq d_1$. For the general case, we just have to observe that the collection of vectors $(H_{0j}(t),...,H_{ij}(t))$ defining $\alpha_{ij}(t)$ is obtained as the in the case $d_0<j\leq d_1$ but pre-multiplying by the appropriate power of $t$.

\qed

We start by noticing that our definitions \eqref{Equation:DefinitionsOfBetas:Equation1}--\eqref{Equation:DefinitionsOfBetas:Equation3} depend on the collection of vectors $(V_{m},...,V_{m'})$. For instance, it is clear that

\begin{align}\beta_{i+1}^{1}(V_{m},...,V_{m'})&=\beta_{i+1}^{0}(V_{m+1},...,V_{m'})\qquad\text{and}\label{Equation:Relations1:Equation1}\\
\hat{\beta}_{i+1}^1(t)(V_{m},...,V_{m'})          &=\hat{\beta}_{i+1}^0(t)(V_{m+1},...,V_{m'})=\beta_{i+1}^0(t)(V_{m+1},...,V_{m'})\quad\text{but in general}\label{Equation:Relations1:Equation2}\\
\beta_{i+1}^1(t)(V_{m},...,V_{m'})&\neq\beta_{i+1}^0(t)(V_{m+1},...,V_{m'}).\label{Equation:Relations1:Equation3}
\end{align}

Since $C_s^i=C_s^{i-1}+\pi_i^\perp C_{s-1}^{i-1}$, one has that
\begin{align}\beta_{i+1}^0&=\beta_{i}^0+\pi_i^\perp\beta_{i}^{1}\quad\text{and}\label{Equation:Relations2:Equation1}\\
\beta_{i+1}^0(t)&=\beta_{i}^0(t)+\pi_{i,t}^\perp\beta_{i}^{1}(t)\quad\text{but in general}\label{Equation:Relations2:Equation2}\\
\hat{\beta}_{i+1}^0(t)&\neq\hat{\beta}_{i}^0(t)+\pi_{i,t}^\perp\hat{\beta}_{i}^{1}(t).\label{Equation:Relations2:Equation3}
\end{align}

\smallskip

We do have the following relation, whose proof follows by manipulation of the expressions defining $\beta_{i+1}^0(t)$, $\hat{\beta}_{i+1}^1(t)$, $\beta_{i+1}^1(t)$ and $\beta_{i+1}^0(t)$:

\begin{lemma}\label{Lemma:LemmaForTheFirstPart} For any collection $(V_m,...,V_{m'})$ we have
\begin{equation}\label{Equation:EquationIn:Lemma:LemmaForTheFirstPart}
t(\beta_{i+1}^0(t)+\hat{\beta}_{i+1}^1(t))-\beta_{i+1}^1(t)=\beta_{i+1}^0(t).
\end{equation}

\end{lemma}

\begin{corollary}\label{Corollary:AfterLemmaForTheFirstPart} For any $i$,
\begin{equation}\label{Equation:EquationIn:Corollary:AfterLemmaForTheFirstPart}
t\pi_{i+1,t}^\perp\big(\beta_{i+1}^0(t)+\hat{\beta}_{i+1}^1(t)\big)=\pi_{i+1,t}^\perp\beta_{i+2}^0(t).
\end{equation}
\end{corollary}
\begin{proof}We have that
\[\begin{array}{ll}t\pi_{i+t,t}^\perp\big(\beta_{i+1}^0(t)+\hat{\beta}_{i+1}^1(t)\big)&=\pi_{i+1,t}^\perp\big(\beta_{i+1}^0(t)+\beta_{i+1}^1(t)\big)\\
~&=\pi_{i+1,t}^\perp\big(\beta_{i+1}^0(t)+\pi_{i+1,t}^\perp\beta_{i+1}^1(t)\big)=\pi_{i+1,t}^\perp\beta_{i+2}^0(t)\text{\tiny{(using \eqref{Equation:Relations2:Equation2})}}.\\
\end{array}\]
\end{proof}

Before proceeding into the proof of Lemma \ref{Lemma:LemmaOfEquation:FirstHypothesisEquationVersionTwo}, we need a final extra notation. We start by noticing that
\[\eta_r^k=\big((\pi_{r,t}+\lambda^{-1}\pi_{r,t}^\perp)\eta_{r-1}(\pi_r+\lambda t\pi_r^\perp)\big)^k\]
Since the first factor has only $\lambda$ powers $0$ or $-1$ and the last one only powers $0$ or $1$, we see that\renewcommand{\arraystretch}{1.3}
\[\begin{array}{ll}
\eta_r^k&=\pi_{r,t}\eta_{r-1}^k\pi_r+t\pi_{r,t}\eta_{r-1}^{k-1}\pi_{r}^\perp+\pi_{r,t}^\perp\eta_{r-1}^{k+1}\pi_{r}+t\pi_{r,t}^\perp\eta_{r-1}^k\pi_r^\perp\\
~&=t(\pi_{r,t}\eta_{r-1}^{k-1}+\pi_{r,t}^\perp\eta_{r-1}^k)\pi_r^\perp+\pi_{r,t}\eta_{r-1}^k\pi_r+\pi_{r,t}^\perp\eta_{r-1}^{k+1}\pi_{r}
\end{array}\]\renewcommand{\arraystretch}{1.0}

We denote by $\eta_r^{k,1}$ the coefficient of $\pi_r^\perp$ in the above expression:
\[\eta_r^{k,1}=t(\pi_{r,t}\eta_{r-1}^{k-1}+\pi_{r,t}^\perp\eta_{r-1}^k)\]
For completeness, we denote $\eta_r^k$ by $\eta_r^{k,0}$ and define recursively for $l\leq k$
\begin{equation}\label{Equation:DefinitionofEtaRKL}
\eta_{r}^{k,l}=t(\pi_{r,t}\eta_{r-1}^{k-1,l-1}+\pi_{r,t}^\perp\eta_{r-1}^{k,l-1}).
\end{equation}
We then have
\begin{lemma}\label{Lemma:GoingUpAndDownOnEtas}For all $j$ and for all $l\leq k\leq j$
\begin{align}
\eta_j^{k,l}\pi_{j-l}&=t^{-1}\eta_{j}^{k+1,l+1}\pi_{j-l}\qquad\text{and}\label{Equation:FirstEquationInLemma:GoingUpAndDownOnEtas}\\
\eta_j^{k,l}\pi_{j-l}^\perp&=\eta_{j}^{k,l+1}\pi_{j-l}^\perp.\label{Equation:SecondEquationInLemma:GoingUpAndDownOnEtas}
\end{align}

\end{lemma}

\textit{Proof of \eqref{Equation:FirstEquationInLemma:GoingUpAndDownOnEtas}.}

We begin by showing that \eqref{Equation:FirstEquationInLemma:GoingUpAndDownOnEtas} is true for all $j$ and for all $k\leq j$ when $l=0$:
\[\eta_j^{k,0}\pi_{j}=t^{-1}\eta_{j}^{k+1,1}\pi_{j}\equi (\pi_j\eta_{j-1}^k+\pi_{j-1}^\perp\eta_{j-1}^{k+1})\pi_{j}=t^{-1}t(\pi_j\eta_{j-1}^{k}+\pi_{j}^\perp\eta_{j-1}^{k+1})\pi_j,\]
as wanted. Next, assume \eqref{Equation:FirstEquationInLemma:GoingUpAndDownOnEtas} is valid for $l$ fixed, for any $k\leq j$ and let us prove that it is also valid for any $k\leq j$ for $l+1$. As a matter of fact, we have\renewcommand{\arraystretch}{1.3}
\[\begin{array}{ll}\eta_j^{k,l+1}\pi_{j-(l+1)}&=t(\pi_j\eta_{j-1}^{k-1,l}+\pi_j^\perp\eta_{j-1}^{k,l})\pi_{j-1-l}=t\pi_j t^{-1}\eta_{j-1}^{k,l+1}\pi_{j-1-l}+t\pi_{j}^\perp t^{-1}\eta_{j-1}^{k+1,l+1}\pi_{j-1-l}\\
~&=(\pi_j \eta_{j-1}^{k,l+1}+\pi_{j}^\perp \eta_{j-1}^{k+1,l+1})\pi_{j-1-l}=t^{-1}\eta_{j}^{k+1,l+1}\pi_{j-(l+1)},
\end{array}\]\renewcommand{\arraystretch}{1.0}

as desired.

\qed

\textit{Proof of \eqref{Equation:SecondEquationInLemma:GoingUpAndDownOnEtas}.} As before, let us prove the case when $l=0$:
\[\eta_j^{k,0}\pi_{j}^\perp=\eta_{j}^{k,1}\pi_{j}^\perp\equi t(\pi_j\eta_{j-1}^{k-1}+\pi_j^\perp\eta_{j-1}^{k})\pi_j^\perp=t(\pi_j\eta_{j-1}^{k-1,0}+\pi_j^\perp\eta_{j-1}^{k,0})\pi_j^\perp,\]
which is true. As for the induction:\renewcommand{\arraystretch}{1.3}
\[\begin{array}{ll}\eta_j^{k,l+1}\pi_{j-(l+1)}^\perp&=t(\pi_j\eta_{j-1}^{k-1,l}+\pi_j^\perp\eta_{j-1}^{k,l})\pi_{j-1-l}=t(\pi_j\eta_{j-1}^{k-1,l+1}+\pi_j^\perp\eta_{j-1}^{k,l+1})\pi_{j-1-l}\\
~&=\eta_{j}^{k,l+2}\pi_{j-(l+1)},
\end{array}\]\renewcommand{\arraystretch}{1.0}

as wanted.

\qed

Now, we introduce the following operators
\begin{equation}\label{Equation:DefinitionOfArk}\renewcommand{\arraystretch}{1.3}
\begin{array}{ll}
A_2^r&=A_2^{r-1}=...=A_2^2=\eta_1^0,\\
A_r^r&=\eta_{r-1}^0+\pi_{r,t}^\perp\eta_{r-1}^1 (r>2)\\
A_3^r&=A_3^{r-1}+t^{3-r}\pi_{r,t}^\perp\eta_{r-1}^{r-2,r-3-1},\\
...&~\\
A_j^r&=A_j^{r-1}+t^{j-r}\pi_{r,t}^\perp\eta_{r-1}^{r-j+1,r-j-1}\quad (j<r)\\
\end{array}
\end{equation}\renewcommand{\arraystretch}{1.0}

\begin{lemma}\label{Lemma:CompleteLemmaToProveLemma:LemmaOfEquation:FirstHypothesisEquationVersionTwo}
For any $r'$ and for any fixed collection $(V_m,..., V_{m+r'})$, the following equation holds
\begin{equation}\label{Equation:ThirdHypothesisEquationVersionTwo}
\eta_{r'-1}^1\beta_{r'}^0+\sum_{s=2}^{r'-1}\pi_{s,t}^\perp\beta_{s+1}^0(t)-t\sum_{s=2}^{r'-1}A_s^{r'-1}\pi_s^\perp\beta_{s+1}^0=-\sum_{s=3}^{r'-1}t^{s-r}\eta_{r'-1}^{r'-s+1,r'-s-1}\pi_{s}^\perp\beta_{s}^0.
\end{equation}
\end{lemma}

\begin{proof}
The equation \eqref{Equation:ThirdHypothesisEquationVersionTwo} shall be proved by induction on $r$. Besides the induction hypothesis, we shall now use the following equations

\begin{align}
\eta_r^1&=\pi_{r,t}\eta_{r-1}^1\pi_r+t\pi_{r,t}\eta_{r-1}^0\pi_r^\perp+\pi_{r,t}^\perp\eta_{r-1}^2\pi_r+t\pi_{r,t}^\perp\eta_{r-1}^1\pi_r^\perp,\label{Equation:AuxiliaryCalculusEq1}\\
A_r^r&=\eta_{r-1}^0+\pi_{r,t}^\perp\eta_{r-1}^1\quad\text{and }\label{Equation:AuxiliaryCalculusEq2}\\
\pi_{r,t}^\perp\beta_{r+1}^0(t)&=\pi_{r,t}^\perp\Big(t\eta_{r-1}^0\pi_r^\perp\beta_{r+1}^0+\eta_{r-1}^1\beta_r^0-\eta_{r-1}^0\pi_r(\beta_{r}^0)+\sum_{s=3}^{r-1}t^{s-r}\eta_{r-1}^{r-s+1,r-s-1}\pi_{s}^\perp\beta_s^0\Big)\label{Equation:AuxiliaryCalculusEq3}
\end{align}

(we shall prove \eqref{Equation:AuxiliaryCalculusEq3} in the end). We have that

\bigskip

$
\begin{array}{ll}
~\quad&\eta_r^1\beta_{r+1}^0+\sum_{s=2}^r\pi_{s,t}^\perp(\beta_{s+1}^0(t))-t\sum_{s=2}^{r}A_s^{r}\pi_s^\perp(\beta_{s+1}^0)\\
\end{array}
$
\smallskip

$
\begin{array}{ll}
=&\displaystyle{(\pi_{r,t}\eta_{r-1}^1\pi_r+t\pi_{r,t}\eta_{r-1}^0\pi_r^\perp+\pi_{r,t}^\perp\eta_{r-1}^2\pi_r+t\pi_{r,t}^\perp\eta_{r-1}^1\pi_r^\perp)\beta_{r+1}^0+\sum_{s=2}^{r-1}\pi_{s,t}^\perp(\beta_{s+1}^0(t))}\\
~&\displaystyle{+\pi_{r,t}^\perp\big(t\eta_{r-1}^0\pi_r^\perp\beta_{r+1}^0+\eta_{r-1}^1\beta_r^0-\eta_{r-1}^0\pi_r(\beta_{r}^0)+\sum_{s=3}^{r-1}t^{s-r}\eta_{r-1}^{r-s+1,r-s-1}\pi_{s}^\perp\beta_s^0\big)}\\
~&\displaystyle{-t\sum_{s=2}^{r-1}A_{s}^{r}\pi_s^\perp(\beta_{s+1}^0)-t(\eta_{r-1}^0+\pi_{r,t}^\perp\eta_{r-1}^1)\pi_r^\perp(\beta_{r+1}^0)}\\
\end{array}
$

$
\begin{array}{ll}
=&t\pi_{r,t}\eta_{r-1}^0\pi_r^\perp\beta_{r+1}^0+t\pi_{r,t}^\perp\eta_{r-1}^1\pi_r^\perp\beta_{r+1}^0+\pi_{r,t}\eta_{r-1}^1\pi_r\beta_{r}^0+\pi_{r,t}^\perp\eta_{r-1}^2\pi_r\beta_{r}^0\\
~&\displaystyle{+\sum_{s=2}^{r-1}\pi_{s,t}^\perp(\beta_{s+1}^0(t))+\pi_{r,t}^\perp\big(t\eta_{r-1}^0\pi_r^\perp\beta_{r+1}^0+\eta_{r-1}^1\beta_r^0-\eta_{r-1}^0\pi_r(\beta_{r}^0)+\sum_{s=3}^{r-1}t^{s-r}\eta_{r-1}^{r-s+1,r-s-1}\pi_{s}^\perp\beta_s^0\big)}\\
~&\displaystyle{-tA_2^{r-1}\pi_2^\perp(\beta_3^0)-t\sum_{s=3}^{r-1}(A_s^{r-1}+t^{s-r}\pi_{r,t}^\perp\eta_{r-s+1}^{r-2,r-s-1})\pi_s^\perp(\beta_{s+1}^0)-t(\eta_{r-1}^0+\pi_{r,t}^\perp\eta_{r-1}^1)\pi_r^\perp(\beta_{r+1}^0)}
\end{array}
$
\smallskip

Now, parcels with $\beta_{r+1}^0$ vanish. Using the induction hypothesis \eqref{Equation:ThirdHypothesisEquationVersionTwo},

$
\begin{array}{ll}
=&\displaystyle{\pi_{r,t}\eta_{r-1}^1\pi_r\beta_{r}^0+\pi_{r,t}^\perp\eta_{r-1}^2\pi_r\beta_{r}^0+\pi_{r,t}^\perp\big(\eta_{r-1}^1\beta_r^0-\eta_{r-1}^0\pi_r(\beta_{r}^0)+\sum_{s=3}^{r-1}t^{s-r}\eta_{r-1}^{r-s+1,r-s-1}\pi_{s}^\perp\beta_{s}^0\big)}\\
~&\displaystyle{-\sum_{s=3}^{r-1}t^{s+1-r}\pi_{r,t}^\perp\eta_{r-1}^{r-s+1,r-s-1}\pi_s^\perp(\beta_{s+1}^0)-\eta_{r-1}^1\beta_{r}^0-\sum_{s=3}^{r-1}t^{s-r}\eta_{r-1}^{r-s+1,r-s-1}\pi_{s}^\perp\beta_{s}^0}
\end{array}
$

Writing $\pi_s^\perp=\text{Id}-\pi_s$ and re-organizing the terms, we end up with

$
\begin{array}{ll}
=&\big(\pi_{r,t}\eta_{r-1}^1\pi_r+\pi_{r,t}^\perp\eta_{r-1}^2\pi_r+\pi_{r,t}^\perp\eta_{r-1}^1-\pi_{r,t}^\perp\eta_{r-1}^0\pi_r-t^{0}\pi_{r,t}^\perp\eta_{r-1}^{2,0}-\eta_{r-1}^1\big)\beta_{r}^0\\
~&+\big(t^{-1}\pi_{r,t}^\perp\eta_{r-1}^{2,0}\pi_{r-1}^\perp-t^{-1}\pi_{r,t}^\perp\eta_{r-1}^{3,1}+t^0\pi_{r,t}^\perp\eta_{r-1}^{2,0}\pi_{r-1}-t^{-1}\eta_{r-1}^{2,0}\pi_{r-1}^\perp\big)\beta_{r-1}^0\\
~&+\big(t^{-2}\pi_{r,t}^\perp\eta_{r-1}^{3,1}\pi_{r-2}^\perp-t^{-2}\pi_{r,t}^\perp\eta_{r-1}^{4,2}+t^{-1}\pi_{r,t}^\perp\eta_{r-1}^{3,1}\pi_{r-2}-t^{-2}\eta_{r-1}^{3,1}\pi_{r-2}^\perp\big)\beta_{r-2}^0\\
~&...\\
~&+\big(t^{4-r}\pi_{r,t}^\perp\eta_{r-1}^{r-3,r-5}\pi_{4}^\perp-t^{4-r}\pi_{r,t}^\perp\eta_{r-1}^{r-2,r-4}+t^{5-r}\pi_{r,t}^\perp\eta_{r-1}^{r-3,r-5}\pi_{4}-t^{4-r}\eta_{r-1}^{r-3,r-5}\pi_{4}^\perp\big)\beta_{4}^0\\
~&+\big(t^{3-r}\pi_{r,t}^\perp\eta_{r-1}^{r-2,r-4}\pi_{3}^\perp-t^{3-r}\pi_{r,t}^\perp\underbrace{\eta_{r-1}^{r-1,r-3}}_{\text{\tiny{$=0$}}}+t^{4-r}\pi_{r,t}^\perp\eta_{r-1}^{r-2,r-4}\pi_{3}-t^{3-r}\eta_{r-1}^{r-2,r-4}\pi_{3}^\perp\big)\beta_{3}^0\\
\end{array}
$

From Lemma \ref{Lemma:GoingUpAndDownOnEtas} (equation \eqref{Equation:FirstEquationInLemma:GoingUpAndDownOnEtas}), we obtain

$
\begin{array}{ll}
=&\displaystyle{-\big(\pi_{r,t}\eta_{r-1}^1+\pi_{r,t}^\perp\eta_{r-1}^{2,0}-0\big)\pi_r^\perp\beta_{r}^0-\sum_{s=3}^{r-1}t^{s-r}\big(\pi_{r,t}\eta_{r-1}^{r-s+1,r-s-1}+\pi_{r,t}^\perp\eta_{r-1}^{r-s+2,r-s}\big)\pi_{s}^\perp\beta_{s}^0}\\
\end{array}
$

Again from Lemma \ref{Lemma:GoingUpAndDownOnEtas} (equation \eqref{Equation:SecondEquationInLemma:GoingUpAndDownOnEtas}) and \eqref{Equation:DefinitionofEtaRKL}

$
\begin{array}{ll}
=&\displaystyle{-\sum_{s=3}^{r}t^{s-r-1}\eta_r^{r-s+2,r-s+1}\pi_s^\perp\beta_s^0\text{\tiny{(use \eqref{Equation:SecondEquationInLemma:GoingUpAndDownOnEtas})}}=-\sum_{s=3}^{r}t^{s-r-1}\eta_r^{r-s+2,r-s}\pi_s^\perp\beta_{s}^0}\\
\end{array}
$

Thus,
\[\eta_r^1\beta_{r+1}^0+\sum_{s=2}^r\pi_{s,t}^\perp(\beta_{s+1}^0(t))-t\sum_{s=2}^{r}A_s^{r}\pi_{s}^\perp(\beta_{s+1}^0)=-\sum_{s=3}^{r}t^{s-1-r}\eta_r^{r-s+2,r-s}\pi_s^\perp\beta_{s}^0\]

\end{proof}

\bigskip

\textit{Proof of Lemma \ref{Lemma:LemmaOfEquation:FirstHypothesisEquationVersionTwo}.}

Using \eqref{Equation:Relations1:Equation1} and \eqref{Equation:Relations1:Equation2}, we notice that for any collection $(V_m,...,V_{m'})$ of vectors,\renewcommand{\arraystretch}{1.3}

\begin{equation}\label{Equation:SecondHypothesisEquationVersionOne}
\begin{array}{ll}\eta_r^0(\beta_{r+1}^1(V_m,...,V_{m'}))&=\hat{\beta}_{r+1}^1(t)(V_m,...,V_{m'})\\
~&+\pi_{2,t}^\perp\big(\hat{\beta}_2^1(t)(V_m,...,V_{m'})\big)+...+\pi_{r,t}^\perp\big(\hat{\beta}_r^1(t)(V_m,...,V_{m'})\big)\\
~&-A_2^{r}\pi_2^\perp\big(\beta_2^1(V_m,...,V_{m'})\big)-...-A_{r}^r\pi_r^\perp\big(\beta_r^1(V_m,...,V_{m'}\big)\end{array}
\end{equation}\renewcommand{\arraystretch}{1.0}

or, more succinctly,

\begin{equation}\label{Equation:SecondHypothesisEquationVersionTwo}
\eta_r^0(\beta_{r+1}^1)=\hat{\beta}_{r+1}^1(t)+\sum_{s=2}^{r}\pi_{s,t}^\perp\hat{\beta}_s^1(t)-\sum_{s=2}^{r}A_s^{r}\pi_s^\perp\beta_s^1.
\end{equation}

We have that
\[\eta_{r+1}=(\pi_{r+1,t}+\lambda^{-1}\pi_{r+1,t}^\perp)\eta_{r}(\pi_{r+1,t}+\lambda t\pi_{r+1,t}^\perp)\]
Since $\eta_r$ has no negative powers of $\lambda$,

\[\eta_{r+1}^0=\pi_{r+1,t}\eta_r^0\pi_{r+1}+\pi_{r+1,t}^\perp\eta_r^1\pi_{r+1}+t\pi_{r+1,t}^\perp\eta_r^0\pi_{r+1}^\perp\]

On the other hand, from \eqref{Equation:FirstHypothesisEquationVersionTwo}, we may conclude that $\eta_{r}^0$ maps $\alpha_{r+1}$ into $\alpha_{r+1}(t)$ so that \[\pi_{r+1,t}^\perp\eta_r^0\pi_{r+1}=0.\]

Now, fix a collection $(V_m,...,V_{m+r+1})$ of $\mathbb{C}^n$ vectors. Then,\renewcommand{\arraystretch}{1.3}

\smallskip

$
\begin{array}{ll}~&\eta_{r+1}^0\big(\beta_{r+2}^0\big)=\big(\pi_{r+1,t}\eta_r^0\pi_{r+1}+\pi_{r+1,t}^\perp\eta_r^1\pi_{r+1}+t\pi_{r+1,t}^\perp\eta_r^0\pi_{r+1}^\perp\big)\big(\beta_{r+1}^0+\pi_{r+1}^\perp\beta_{r+1}^1\big)\,\text{\tiny{(see \eqref{Equation:Relations2:Equation1})}}\\
=&\pi_{r+1,t}\eta_r^0\pi_{r+1}(\beta_{r+1}^0)+\pi_{r+1,t}^\perp\eta_r^1\pi_{r+1}(\beta_{r+1}^0)+t\pi_{r+1,t}^\perp\eta_r^0\pi_{r+1}^\perp(\beta_{r+1}^0+\beta_{r+1}^1)\\
=&\eta_r^0\pi_{r+1}(\beta_{r+1}^0)-\pi_{r+1,t}^\perp\eta_r^1\pi_{r+1}^\perp(\beta_{r+1}^0)+\pi_{r+1,t}^\perp\eta_r^1(\beta_{r+1}^0)+t\pi_{r+1,t}^\perp\eta_r^0(\beta_{r+1}^0+\beta_{r+1}^1)\\
=&\eta_r^0(\beta_{r+1}^0)-(\eta_r^0+\pi_{r+1,t}^\perp\eta_r^1)\pi_{r+1}^\perp(\beta_{r+1}^0)+\pi_{r+1,t}^\perp\eta_r^1(\beta_{r+1}^0)+t\pi_{r+1,t}^\perp\eta_r^0(\beta_{r+1}^0+\beta_{r+1}^1)\\
=&\displaystyle{\beta_{r+1}^0(t)+\sum_{s=2}^r\pi_{s,t}^\perp\beta_s^0(t)-\sum_{s=2}^{r}A_s^{r}\pi_s^\perp\beta_s^0-A_{r+1}^{r+1}\pi_{r+1}^\perp\beta_{r+1}^0}+\pi_{r+1,t}^\perp\big(\eta_r^1\beta_{r+1}^0+t\eta_r^0(\beta_{r+1}^0+\beta_{r+1}^1)\big)
\end{array}
$\renewcommand{\arraystretch}{1.0}

\smallskip
where we have used the induction hypothesis. We can proceed to get
\smallskip

$
\begin{array}{ll}
=&\displaystyle{\beta_{r+1}^0(t)+\sum_{s=2}^r\pi_{s,t}^\perp\beta_s^0(t)-\sum_{s=2}^{r}A_s^{r}\pi_s^\perp\beta_s^0-A_{r+1}^{r+1}\pi_{r+1}^\perp\beta_{r+1}^0}\\
~&+\pi_{r+1,t}^\perp\big(\eta_r^1\beta_{r+1}^0+t\eta_r^0(\beta_{r+1}^0+\beta_{r+1}^1)\big)-\pi_{r+1,t}^\perp\beta_{r+1}^1(t)+\pi_{r+1,t}^\perp\beta_{r+1}^1(t)
\end{array}
$

\smallskip

Now, using \eqref{Equation:Relations2:Equation2}, \eqref{Equation:FirstHypothesisEquationVersionTwo}) and \eqref{Equation:SecondHypothesisEquationVersionTwo}, this expression becomes

\smallskip

$
\begin{array}{ll}
=&\displaystyle{\beta_{r+2}^0(t)+\sum_{s=2}^r\pi_{s,t}^\perp\beta_s^0(t)-\sum_{s=2}^{r}A_s^{r}\pi_s^\perp\beta_s^0-A_{r+1}^{r+1}\pi_{r+1}^\perp\beta_{r+1}^0}\\
~&\displaystyle{+\pi_{r+1,t}^\perp\big(\eta_r^1\beta_{r+1}^0+t\big(\beta_{r+1}^0(t)+\hat{\beta}_{r+1}^1(t)+\sum_{s=2}^r\pi_{s,t}^\perp(\beta_s^0(t)+\hat{\beta}_s^1(t))}\\
~&\displaystyle{-\sum_{s=2}^{r}A_s^{r}\pi_s^\perp(\beta_s^0+\beta_s^1)\big)-\beta_{r+1}^1(t)\big)}\\
\end{array}
$

\smallskip

Using \eqref{Equation:EquationIn:Lemma:LemmaForTheFirstPart}, \eqref{Equation:EquationIn:Corollary:AfterLemmaForTheFirstPart} and \eqref{Equation:Relations2:Equation1} we get

\smallskip

$
\begin{array}{ll}
=&\displaystyle{\beta_{r+2}^0(t)+\sum_{s=2}^r\pi_{s,t}^\perp\beta_s^0(t)-\sum_{s=2}^{r}A_s^{r}\pi_s^\perp\beta_s^0-A_{r+1}^{r+1}\pi_{r+1}^\perp\beta_{r+1}^0}\\
~&\displaystyle{+\pi_{r+1,t}^\perp\big(\eta_r^1\beta_{r+1}^0+\beta_{r+1}^0(t)+\sum_{s=2}^r\pi_{s,t}^\perp(\beta_{s+1}^0(t))-t\sum_{s=2}^{r}A_s^{r}\pi_s^\perp(\beta_{s+1}^0)\big)}\\
=&\displaystyle{\beta_{r+2}^0(t)+\sum_{s=2}^r\pi_{s,t}^\perp\beta_s^0(t)+\pi_{r+1,t}^\perp\beta_{r+1}^0(t)-\sum_{s=2}^{r}A_s^{r}\pi_s^\perp\beta_s^0-A_{r+1}^{r+1}\pi_{r+1}^\perp\beta_{r+1}^0}\\
~&\displaystyle{+\pi_{r+1,t}^\perp\big(\eta_r^1\beta_{r+1}^0+\sum_{s=2}^r\pi_{s,t}^\perp(\beta_{s+1}^0(t))-t\sum_{s=2}^{r}A_s^{r}\pi_s^\perp(\beta_{s+1}^0)\big)}\\
\end{array}
$

\smallskip

Using \eqref{Equation:ThirdHypothesisEquationVersionTwo},

\smallskip

$
\begin{array}{ll}
=&\displaystyle{\beta_{r+2}^0(t)+\sum_{s=2}^{r+1}\pi_{s,t}^\perp\beta_s^0(t)-\sum_{s=2}^{r}A_s^{r}\pi_s^\perp\beta_s^0-A_{r+1}^{r+1}\pi_{r+1}^\perp\beta_{r+1}^0+\pi_{r+1,t}^\perp\big(-\sum_{s=3}^{r}t^{s-1-r}\eta_r^{r-s+2,r-s}\pi_s^\perp\beta_{s}^0\big)}\\
\end{array}
$

\smallskip

and, re-organizing terms:

\smallskip

$
\begin{array}{ll}
=&\displaystyle{\beta_{r+2}^0(t)+\sum_{s=2}^{r+1}\pi_{s,t}^\perp\beta_s^0(t)-A_2^{r}\pi_2^\perp\beta_2^0}\\
~&-(A_3^r+t^{2-r}\pi_{r+1,t}^\perp\eta_r^{r-1,r-3})\pi_3^\perp\beta_3^0-...-(A_{r}^r+t^{-1}\pi_{r+1,t}^\perp\eta_r^{2,0})\pi_r^\perp\beta_r^0-A_{r+1}^{r+1}\pi_{r+1}^\perp\beta_{r+1}^0\\
=&\displaystyle{\beta_{r+2}^0(t)+\sum_{s=2}^{r+1}\pi_{s,t}^\perp\beta_s^0(t)-\sum_{s=2}^{r+1}A_s^{r+1}\pi_s^\perp\beta_s^0},\\
\end{array}
$

\smallskip

completing the proof.

\qed

We conclude by proving \eqref{Equation:AuxiliaryCalculusEq3}:

Since we may use \eqref{Equation:FirstHypothesisEquationVersionTwo} applied to $\beta_r^0$, we can write
\smallskip

$
\begin{array}{ll}
~&\displaystyle{\pi_{r,t}^\perp\beta_{r+1}^0(t)=\pi_{r,t}^\perp\big(\eta_r^0(\beta_{r+1}^0)-\sum_{s=2}^r\pi_{s,t}^\perp\beta_s^0(t)+\sum_{s=2}^{r}A_s^{r}\pi_s^\perp\beta_s^0\big)}\\

\end{array}
$
\smallskip

Using \eqref{Equation:DefinitionOfArk},

\smallskip

$
\begin{array}{ll}
=&\displaystyle{\pi_{r,t}^\perp\big(\eta_r^0(\beta_{r+1}^0)-\sum_{s=2}^{r-1}\pi_{s,t}^\perp\beta_s^0(t)-\pi_{r,t}^\perp(\eta_{r-1}^0(\beta_{r}^0)-\sum_{s=2}^{r-1}\pi_{s,t}^\perp\beta_s^0(t)+\sum_{s=2}^{r-1}A_s^{r-1}\pi_s^\perp\beta_s^0)+}\\
~&\displaystyle{+A_2^{r-1}\pi_2^\perp\beta_2^0+\sum_{s=3}^{r-1}(A_s^{r-1}+t^{s-r}\pi_{r,t}^\perp\eta_{r-1}^{r-s+1,r-s-1})\pi_{s}^\perp\beta_s^0+(\eta_{r-1}^0+\pi_{r,t}^\perp\eta_{r-1}^1)\pi_r^\perp\beta_r^0\big)}\\
\end{array}
$

$
\begin{array}{ll}
=&\displaystyle{\pi_{r,t}^\perp\big(\eta_r^0(\beta_{r+1}^0)-\eta_{r-1}^0(\beta_{r}^0)+\sum_{s=3}^{r-1}t^{s-r}\eta_{r-1}^{r-s+1,r-s-1}\pi_{s}^\perp\beta_s^0+(\eta_{r-1}^0+\eta_{r-1}^1)\pi_r^\perp\beta_r^0\big)}\\
=&\displaystyle{\pi_{r,t}^\perp\big((\eta_{r-1}^1\pi_r+t\eta_{r-1}^0\pi_r^\perp)\beta_{r+1}^0-\eta_{r-1}^0(\beta_{r}^0)+\sum_{s=3}^{r-1}t^{s-r}\eta_{r-1}^{r-s+1,r-s-1}\pi_{s}^\perp\beta_s^0+(\eta_{r-1}^0+\eta_{r-1}^1)\pi_r^\perp\beta_r^0\big)}\\
\end{array}
$

$
\begin{array}{ll}
=&\displaystyle{\pi_{r,t}^\perp\big(t\eta_{r-1}^0\pi_r^\perp\beta_{r+1}^0+\eta_{r-1}^1\pi_r\beta_r^0-\eta_{r-1}^0(\beta_{r}^0)+\sum_{s=3}^{r-1}t^{s-r}\eta_{r-1}^{r-s+1,r-s-1}\pi_{s}^\perp\beta_s^0+(\eta_{r-1}^0+\eta_{r-1}^1)\pi_r^\perp\beta_r^0\big)}\\
=&\displaystyle{\pi_{r,t}^\perp\big(t\eta_{r-1}^0\pi_r^\perp\beta_{r+1}^0+\eta_{r-1}^1\beta_r^0-\eta_{r-1}^0\pi_r(\beta_{r}^0)+\sum_{s=3}^{r-1}t^{s-r}\eta_{r-1}^{r-s+1,r-s-1}\pi_{s}^\perp\beta_s^0\big)}.
\end{array}
$

\qed


\end{document}